\documentclass[11pt,oneside,reqno]{amsart}
\usepackage[latin1]{inputenc}
\usepackage[english]{babel}
\usepackage{mathrsfs}
\usepackage{indentfirst}
\usepackage{paralist}
\usepackage{amssymb}
\usepackage{amsthm}
\usepackage{amsmath}
\usepackage{amscd}
\usepackage{graphicx}
\usepackage[colorlinks=true]{hyperref}
\usepackage[margin=1in]{geometry}
\usepackage{cite}
\usepackage{lineno} %% adds numbers to the lines
\usepackage{enumitem} %% needed for options in 'enumerate'
\usepackage{multicol}
\usepackage{xcolor}
\definecolor{ForestGreen}{rgb}{0.13, 0.55, 0.13}

\theoremstyle{plain} 
\newtheorem{theorem}{Theorem}[section] 
\newtheorem*{theorem*}{Main Result}
\newtheorem{thm*}{Known result}
\newtheorem{corollary}[theorem]{Corollary} 
\newtheorem{lemma}[theorem]{Lemma}
\newtheorem{proposition}[theorem]{Proposition}
\newtheorem{definition}[theorem]{Definition}

\theoremstyle{definition}
\newtheorem{example}[theorem]{Example}

\theoremstyle{remark}
\newtheorem{remark}[theorem]{Remark}

\numberwithin{equation}{section}

\newcommand{\eqlab}[1]{\begin{equation}  \begin{aligned}#1 \end{aligned}\end{equation}} %this substitutes \begin{equation}\begin{aligned} ... 
\newcommand{\bgs}[1]{\begin{equation*} \begin{aligned}#1\end{aligned}\end{equation*}} %this substitutes \begin{equation*}\begin{aligned} ...
\newcommand{\syslab}[2] []  {\begin{equation}#1  \left\{\begin{aligned}#2\end{aligned}\right.\end{equation}} %this substitutes \begin{equation}\begin{aligned} ... with the graph parenthesis 
\newcommand{\sys}[2][]{\begin{equation*}#1  \left\{\begin{aligned}#2\end{aligned}\right.\end{equation*}}%this substitutes \begin{equation*}\begin{aligned} ... with the graph parenthesis 

\def\z{{\bf z}}

\newcommand{\R}{\ensuremath{\mathbb{R}}}
\newcommand{\Rn}{\ensuremath{\mathbb{R}^n}}
\newcommand{\N}{\ensuremath{\mathbb{N}}}
\newcommand{\eps}{\ensuremath{\varepsilon}}
\newcommand{\OMega}{\ensuremath{\Omega}}

\newcommand{\Co}{\mathcal C}
\newcommand{\E}{\mathcal E_s}
\newcommand{\En}{\mathcal E}
\newcommand{\Ha}{\mathcal H}

\newcommand{\Op}{\mathcal O}
\newcommand{\W}{\mathcal W}

\DeclareMathOperator{\Per}{Per}
\DeclareMathOperator{\sgn}{sgn}
\DeclareMathOperator{\loc}{loc}
\DeclareMathOperator{\diam}{diam}
\DeclareMathOperator{\dist}{dist}
\DeclareMathOperator{\Ts}{Tail}
\newcommand{\Tss}{\ensuremath{\Ts_s}}

\renewcommand{\le}{\leqslant}
\renewcommand{\leq}{\leqslant}
\renewcommand{\ge}{\geqslant}
\renewcommand{\geq}{\geqslant}

\title[$(s,p)$-harmonic approximation of functions of least $W^{s,1}$-seminorm]{$(s,p)$-harmonic approximation \\ of functions of least $W^{s,1}$-seminorm}
\author[C. Bucur]{Claudia Bucur}
\author[S. Dipierro]{Serena Dipierro}
\author[L. Lombardini]{Luca Lombardini}
\author[J. M. Maz\'{o}n]{Jos\'{e} M. Maz\'{o}n}
\author[E. Valdinoci]{Enrico Valdinoci}

\address{Claudia Bucur \textsuperscript{1}}
\address{Serena Dipierro \textsuperscript{2}}
\address{Luca Lombardini \textsuperscript{2,3}}
\address{Jos\'{e} M. Maz\'{o}n \textsuperscript{4}}
\address{Enrico Valdinoci \textsuperscript{2}}

\address{{\textsuperscript{1}
Universit\`a degli Studi dell'Insubria,
Dipartimento di Scienza e Alta Tecnologia,
and RISM-Riemann International School of Mathematics,
Villa Teoplitz, Via G. B. Vico 46, 21100 Varese, Italy}}

\address{{\textsuperscript{2}
University of Western Australia,
Department of Mathematics and Statistics,
35 Stirling Highway,
Crawley, Perth,
WA6009, Australia}}

\address{{\textsuperscript{3}
Universit\`a degli Studi di Padova, 
Dipartimento di Matematica ``Tullio Levi-Civita'', 
Via Trieste 63, 35121 Padova, Italy}}

\address{{\textsuperscript{4}
Universitat de Val\`encia,
Facultat de Ci\`encies Matem\`atiques,
Dr. Moliner 50,
46100 Burjassot, Val\`encia, Spain}}

\email{claudiadalia.bucur@uninsubria.it}
\email{serena.dipierro@uwa.edu.au}
\email{luca.lombardini@uwa.edu.au}
\email{Jose.M.Mazon@uv.es}
\email{enrico.valdinoci@uwa.edu.au}

\thanks{The first author is member of GNAMPA (INdAM), Italy.
The second and fifth authors are
members of GNAMPA (INdAM), Italy, and AustMS, Australia.
The second author is supported by
the Australian Research Council DECRA DE180100957
``PDEs, free boundaries and applications''. The fourth author is partially supported by the Spanish
MCIU and FEDER, project PGC2018-094775-B-100. The fifth author is supported by
the Australian Laureate Fellowship
FL190100081
``Minimal surfaces, free boundaries and partial differential equations''.}


\begin{document}

\begin{abstract}
We investigate the convergence as~$p\searrow1$ of the minimizers of the $W^{s,p}$-energy for~$s\in(0,1)$
and~$p\in(1,\infty)$ to those of the $W^{s,1}$-energy, both in the pointwise sense and by means of $\Gamma$-convergence. We also
address the convergence of the corresponding Euler-Lagrange equations, and the equivalence between minimizers and weak solutions. 
As ancillary results,  we study some  regularity issues regarding
minimizers of the $W^{s,1}$-energy.
\end{abstract}
\maketitle

\setcounter{tocdepth}{1}
\tableofcontents

\section{Introduction}

The goal of this paper is to study the limit properties of the minimizers of a nonlinear nonlocal problem in dependence of its nonlinear exponent. Roughly speaking, we take into consideration the minimizers of a $W^{s,p}$-Gagliardo seminorm with~$s\in(0,1)$ and~$p>1$ and discuss the limit as~$p\searrow1$. This asymptotic study is important in providing a coherent setting for variational problems for energy functionals that are homogeneous of degree one and not strictly convex.
We try to keep the presentation of the results as self-contained as
possible, so that the paper can be accessible also to a non-specialist reader.

The classical counterpart of this problem is related to isotropic diffusion models restricted on level surfaces as well as to hypersurfaces with zero mean curvature. The corresponding local energy functional is the seminorm in the space of bounded variation functions, whose minimizers are often called ``functions of least gradient'': more specifically, these problems are modeled by the $1$-Laplace operator, and a very fruitful field of investigation consists in understanding the limit of the solutions of $p$-Laplace
equations as~$p\searrow1$, see~\cite{MR2164415}. An evident structural difficulty in this setting is to give an appropriate meaning to the $1$-Laplace operator, or even to the normal vector field~$\frac{\nabla u}{|\nabla u|}$, at points where~$\nabla u$ vanishes.
In these classical problems,
the appropriate substitute for the normal vector field at critical points was
introduced by
Andreu, Ballester, Caselles and Maz\'{o}n in~\cite{MR1814993}
via a suitable vector field~$\mathbf{z}$ with~$|\mathbf{z}|\le1$ and~$\mathbf{z}\cdot\nabla u=|\nabla u|$. 
We refer to~\cite{mazrossleon}
for equivalence results between
functions of least gradient and solutions of $1$-Laplace equations.
See also~\cite{MR2376662}
and the references therein
for several motivations and perspectives related to nonlinear PDEs involving the $1$-Laplacian.
\medskip

The nonlocal correspondent of these classical problems entails additional difficulties, since the role played by the normal vector is taken in this setting by the fractional ratio~$\frac{u(x)-u(y)}{|u(x)-u(y)|}$
and hence the singular set is geometrically more difficult to interpret and describe. As a counterpart, finding a suitable substitute of this ratio that carries over to the singularities is conceptually more difficult than in the classical case, and a first step towards
the understanding of this problem was made in~\cite{toled}
where the nonlocal ratio was replaced by a convenient choice of a measurable function.\medskip

Our objective is to further understand the nonlocal $1$-Laplace equation in view of some convenient limit properties of $p$-Laplace equations as~$p\searrow1$, especially in light of the convergence of the minimizers, of the $\Gamma$-convergence and of the convergence of the weak solutions.

For completeness,  we investigate also the asymptotics as $s \to 1^-$ having  fixed $p=1$, proving convergence of the $W^{s,1}$-energy and of its corresponding minimizers to their local counterparts, i.e. the $BV$ seminorm and respectively,  functions of least gradient. These side results are the content of the Appendix. 
  
  \medskip
  
      To state precisely our results,
we introduce now the formal mathematical setting that we consider in this paper.

We denote by~$\Omega \subset \Rn$ a bounded open set with Lipschitz boundary.
We consider also $s\in(0,1)$ and~$p\in (1,\infty)$. In particular, since our concern is the asymptotic behavior as $p\searrow  1$, we will take $p$ as close to $1$ as needed. 

For any measurable function~$u\colon \Rn \to \R$ and $q\in [1,\infty)$ we define the \emph{nonlocal} $(s,q)$--energy of $u$ in a domain $\Omega$ as
	\[
		\E^q (u) := \frac1{2q}\iint_{Q(\Omega)} \frac{ |u(x)-u(y)|^q}{|x-y|^{n+sq}} \, dx \, dy,
	\]
	where
	$$Q(\Omega):= \R^{2n} \setminus (\Co \Omega)^2.$$
	Notice that one can split $\E^q$ into the contributions occurring inside $\Omega$ and the interactions of $\Omega$ with its complement, precisely
	\eqlab{ \label{split}
		\E^q(u)= \frac1{2q} \int_{\Omega} \int_\Omega \frac{ |u(x)-u(y)|^q}{|x-y|^{n+sq}} \, dx \, dy +  \frac1{q}\int_{\Omega} \int_{\Co \Omega} \frac{ |u(x)-u(y)|^q}{|x-y|^{n+sq}} \, dx \, dy
	.
	} 
In this notation, for the sake of simplicity,
we neglect the dependence on $\Omega$ in the expression of the energy $\E^q$
since the domain $\Omega$ will be fixed throughout the paper---unless otherwise specified.

	We recall  that the fractional $(s,q)$--Gagliardo seminorm  of a measurable function $u\colon\Omega\to\R$ is defined as
	\[
		 \; [u]_{W^{s,q}(\Omega)} :=\left( \int_{\Omega} \int_\Omega \frac{ |u(x)-u(y)|^q}{|x-y|^{n+sq}} \, dx \, dy\right)^{\frac{1}{q}}.
	\] 
We consider the fractional Sobolev space 
	\[ W^{s,q}(\Omega) = \{ u \in L^q(\Omega) \ | \ [u]_{W^{s,q}(\Omega)} < \infty \},
	\]
which is a Banach space  with respect to the norm
	\[
	\Vert u \Vert_{W^{s,q}(\Omega)}:= [u]_{W^{s,q}(\Omega)} + \Vert u \Vert_{L^q(\Omega)}.\] 
	For details on fractional spaces, see for instance~\cite{hitch}.
	\medskip

	We introduce  the functional spaces in which we look for minimizers of the energy $\E^q$, with given Dirichlet data.  As customary in nonlocal problems, the ``boundary'' condition is actually an exterior condition, that is, for the minimizing problem we fix an exterior data $\varphi \colon \Co \Omega \to \R$. We define 
	\eqlab{ \label{dgggg}
		& \W^{s,q}(\Omega):= \left\{ u\colon \Rn \to \R \mbox{ measurable} \; \; \big| \; \;  u |_{\Omega} \in W^{s,q}(\Omega) \right\}\qquad \mbox{ and} \\
		& \W^{s,q}_\varphi(\Omega) := \left\{ u \in \W^{s,q}(\Omega), \; \; \big| \; \; u= \varphi \; \mbox{ a.e. in } \Co \Omega\right\}.  		 
	}
	
	\begin{definition}\label{DEF-1}
We say that $u\in \W^{s,q}(\Omega)$ is an $(s,q)$-minimizer in $\Omega$ if $\E^q(u)<\infty$ and
	\[
		\E^q(u) \leq \E^q(v)
	\]
	for all $v\in \W^{s,q}(\Omega)$ such that $v=u$ almost everywhere in $\Co \Omega$.\end{definition}
	
When $q=1$ we can consider a more general definition of minimizer, that was introduced and studied in \cite{bdlv20}.

\begin{definition}\label{DEF-2}
We say that $u\in \W^{s,1}(\Omega)$ is an $s$-minimal function in $\Omega$ if
	\begin{equation}\label{RENO}
	\iint_{Q(\Omega)}\left[\frac{|u(x)-u(y)|}{|x-y|^{n+s}}-\frac{|v(x)-v(y)|}{|x-y|^{n+s}}\right]dx\,dy\le0,
	\end{equation}
	for all $v\in \W^{s,1}(\Omega)$ such that $v=u$ almost everywhere in $\Co \Omega$.
\end{definition}
We remark that this definition is always well-posed
with no conditions on $u|_{\Co\Omega}$, thanks to the fractional
Hardy-type inequality (see \cite[formula~(17)]{Dyda} - recalled here in \eqref{fracH}).
We also point out that the minimization property in~$\Omega$
induces a minimization in any subdomain~$\Omega'\subset\Omega$
(see e.g. the observation below equation~(1.1)
in~\cite{MR3133422}).

Throughout this paper, it will be useful to consider the ``nonlocal tail'' of a function~$u$.
Namely, as in \cite{teolu},
for  any $q\in[1, \infty)$, one defines
\eqlab{ \label{coda1}
	\Tss^q(u,\Co \Omega; x) :=\int_{\Co \Omega} \frac{|u(y)|^q}{|x-y|^{n+sq}} dy.	
	}
	
	Notice that an $(s,1)$-minimizer is also an $s$-minimal function, and, when $\Tss^1 (u,\Co \Omega; \cdot) \in L^1(\Omega)$, the two notions of minimizer coincide (see \cite[Lemma 2.1]{bdlv20}).

This notation of nonlocal tail also sheds some light on the relation between Definitions~\ref{DEF-1}
and~\ref{DEF-2}. In particular, the assumption~$\E^1(u)<\infty$ is not needed 
	in Definitions~\ref{DEF-2}  since the integrand in~\eqref{RENO} already provides the necessary
	cancellations.
Regarding Definition~\ref{DEF-1},
we stress that the condition $\E^q(u)<\infty$ implies that the exterior datum $u|_{\Co\Omega}$
satisfies an appropriate integrability condition for the tail, according to the following result. 

		\begin{lemma}\label{domain_lem}
Let $q\in[1,1/s)$ and let $u:\R^n\to\R$ be a measurable function. Then,
\[
\E^q(u)<\infty\quad\mbox{if and only if}\quad u\in \W^{s,q}(\Omega)\mbox{ and }\Tss^q (u,\Co \Omega; \cdot) \in L^1(\Omega).
\]
\end{lemma}
Concerning the hypothesis in Lemma~\ref{domain_lem},
we point out that it is not necessary to suppose a-priori
that $u\in L^1_{\loc}(\Rn)$ in order to define the functional $\E^q$ (which might as well be infinite).
However, it turns out indeed that a measurable function with finite $(s,q)$-energy
belongs $L^1_{\loc}(\Rn)$, see also Remark \ref{aggiunta1}.

\medskip

We now discuss the equation arising from the minimization problems
presented in Definitions~\ref{DEF-1} and~\ref{DEF-2}.
For $p>1$, the Euler-Lagrange equation of the $(s,p)$--energy gives rise to the fractional $p$-Laplacian. 
We recall that, formally, the {\it fractional $p$-Laplace operator} is defined as
$$(-\Delta)_p^s u(x):=  P.V. \int_{\R^n} \frac{\vert u(x) - u(y) \vert^{p-2}(u(x) - u(y))}{\vert x - y \vert^{n+sp}} dy.$$

\begin{definition} \label{D:p}
Let~$\varphi \colon \Co \Omega \to \R$ be such that
\begin{equation}\label{ugv7ryfd}
\Tss^p(\varphi,\Co \Omega; \cdot) \in L^1(\Omega),\end{equation}
	 for some $p\in (1,1/s)$.
We say that
a measurable function $u\colon \Rn \to \R$
is a weak solution to the problem  
\syslab[]{ \label{Dirichlet1} 
 & (-\Delta)_p^s u =0 & &\hbox{ in } \; \Omega,
 \\ 
 & u = \varphi & &\hbox{ in }\;  \Co \Omega,
 }
if $u \in \W_\varphi^{s,p}(\Omega)$ and
\eqlab {\label{huf2}
	\iint_{Q(\Omega)}  \frac{1}{\vert x - y \vert^{n+sp}} \vert u(x) - u(y) \vert^{p-2}(u(x) - u(y))(w(x) - w(y)) dx\, dy =0,}
for every $w \in  \W^{s,p}_0(\Omega)$. 
\end{definition}

Notice that, under assumption~\eqref{ugv7ryfd}, Lemma~\ref{domain_lem} ensures that $\En_s^p(u)<\infty$, hence \eqref{huf2} is well-posed. Indeed, since $sp<1$, by 
H\"{o}lder's inequality and the fractional Hardy-type inequality \eqref{fracH} we have
\eqlab{\label{Vivien}
&\left|\iint_{Q(\Omega)}  \frac{1}{\vert x - y \vert^{n+sp}} \vert u(x) - u(y) \vert^{p-2}(u(x) - u(y))(w(x) - w(y)) dx\, dy\right|\\
&\qquad\leq (\En_s^p(u))^\frac{p-1}{p}(\En_s^p(w))^\frac{1}{p}\leq C(n,s,p,\Omega)(\En_s^p(u))^\frac{p-1}{p}\|w\|_{W^{s,p}(\Omega)}.
}
We also remark that, by the density of $C^\infty_c(\Omega)$ in $W^{s,p}(\Omega)$ 
and \eqref{Vivien}, we can consider as test functions in \eqref{huf2}
just $w\in C^\infty_c(\Omega)$.\medskip

When~$p=1$, 
the term~$\vert u(x) - u(y) \vert^{p-2}(u(x) - u(y))$ in \eqref{huf2}
reduces to~$\frac{
u(x) - u(y)}{\vert u(x) - u(y) \vert}$ and
it is evidently problematic to give a rigorous meaning to this ratio.
For this, recalling~\cite{toled},
we give the next definition of weak solution for the $(s,1)$-Laplacian.

\begin{definition}\label{s1lapdel}
	{
		We say that a measurable function $u:\R^n\to\R$ is a weak solution to the problem
\begin{equation}\label{armoni} (-\Delta)_1^s u = 0 \quad \hbox{in} \ \Omega,
\end{equation}
if there exists $\z \in L^\infty(Q(\Omega))$, with $\Vert \z \Vert_{L^\infty(Q(\Omega))} \leq 1$,
$\z(x,y)=-\z(y,x)$,
\begin{equation}\label{ee1}\iint_{Q(\Omega)} \frac{\z(x,y)}{ \vert x - y \vert^{n+s}} (w(x) - w(y)) dx\, dy =0 \quad \hbox{for all} \ w \in \W^{s,1}_0(\Omega),\end{equation}
and
\begin{equation}\label{ee2}\z(x,y) \in \sgn(u(x) - u(y)) \quad \hbox{for almost all} \ (x, y) \in Q(\Omega).
\end{equation}
}
If, in addition,~$u \in \W_\varphi^{s,1}(\Omega)$, we say that~$u$ is a weak solution of the problem
	\syslab[]{\label{Dirichlet1ss}
 & (-\Delta)_1^s u =0 & &\hbox{ in } \; \Omega,
 \\ 
 & u = \varphi & &\hbox{ in }\;  \Co \Omega.
 }
\end{definition}

Concerning the notation used in~\eqref{ee2},
we recall that $\sgn (x)$ denotes a generalized sign function, satisfying 	\[ 
\sgn (x) \, x = |x| \qquad{\mbox{and}}\qquad \sgn(0)= [-1,1].\]
In this setting, equation~\eqref{ee2} translates into
\[ \z(x,y) (u(x)-u(y)) =|u(x)-u(y)|.\]  
In a sense, Definition~\ref{s1lapdel} (as developed in~\cite{toled} for the fractional case), can be seen as a natural counterpart of the setting
presented in~\cite{MR1814993} and~\cite{mazrossleon} for the $1$-Laplace equation.
	\medskip
	
We now focus on the main results of this paper, namely
we study the asymptotics as $p\searrow 1$ of nonlocal
$(s,p)$--problems to the corresponding $(s,1)$--problems. This aim is threefold, and is articulated in:
	\begin{enumerate}
	\item the convergence of minimizers of the $(s,p)$--energy,
	\item the $\Gamma$-convergence of the $(s,p)$--energy,
	\item the convergence of weak solutions of the $(s,p)$--Laplacian.
	\end{enumerate}
We now describe in further details the principal results
that we give here, according to each of these three lines of research.	
	\medskip
	
	The main result related to point (1) goes as follows:
	
	\begin{theorem}\label{theorem}
Let $p_k \searrow 1$ as $k\to\infty$, $\varphi_k \colon \Co \Omega \to \R$ be such that 
\eqlab{ \label{uui3}
   \sup_{k\in\N} \Tss^{p_k}(\varphi_k,\Co \Omega; \cdot) \in {L^1(\Omega)} , \qquad  \qquad		\varphi_k \xrightarrow[k\to \infty]{} \varphi  \mbox{ a.e. in } \Co \Omega,
	}
	and $u_{p_k} \in \W^{s,p_k}_{\varphi_k}(\Omega)$  be a sequence of $(s,p_k)$-minimizers. 
	
	Then, there exist a subsequence $p_{k_j} \searrow 1$ and $u_1 \in \W^{s,1}_\varphi(\Omega)$ such that 
	\[
		u_{p_{k_j}} \xrightarrow[j \to \infty]{} u_1 \qquad \mbox{ in } \;  L^1(\Omega)\quad\mbox{and a.e. in }\R^n.
	\] 
Furthermore, $\Tss^1(\varphi,\Co\Omega;\cdot)\in L^1(\Omega)$ and $u_1$ is an $(s,1)$-minimizer.
\end{theorem}

Similar results related to Theorem~\ref{theorem}
have been also discussed in~\cite[Section 5]{BN18}.

The proof of Theorem~\ref{theorem} relies on ``direct methods'', based on compactness properties and uniform bounds of minimizers.
An alternative approach to this type of questions can be taken in light of the
$\Gamma$-convergence theory, leading to the research direction presented in point~(2).
To follow this line of investigation, we define, for any $q\in [1,\infty)$,
	\begin{equation}\label{kg-98uhn-jmnMSA-1}
		\mathcal X^{q}(\Omega):= \left\{ u \in L^1_{\loc}(\Rn) \; \bigg| \; \iint_{Q(\Omega)} \frac{ |u(x)-u(y)|^{{q}}}{|x-y|^{n+s{q}}} dx \, dy <\infty\right\},   
\end{equation}
	and we introduce the (extended) functional on the space $L^1_{\loc}(\Rn)$ defined by
\begin{equation}\label{kg-98uhn-jmnMSA-2}
\tilde\En_s^{q} (u):=\begin{cases}
		\displaystyle\frac{1}{2q}\iint_{Q(\Omega)}\displaystyle\frac{ |u(x)-u(y)|^{{q}}}{|x-y|^{n+s{q}}} dx \, dy  & \mbox{ if } \; u \in \mathcal X^{q} (\Omega),
		\\
		 +\infty &\mbox{ if } \; u \in L^1_{\loc}(\Rn) \setminus \mathcal X^{q} (\Omega).
		\end{cases}\end{equation}
The main result related to point (2) is the following
$\Gamma$-convergence result in the $L^1_{\loc}(\Rn)$-topology.

	\begin{theorem}\label{gammy}
	We have that
		\bgs{
		{\Gamma\mbox{-}\lim_{p \searrow 1} }\tilde \En_s^{p} = \tilde\En_s^1,
	}
in the $L^1_{\loc}(\Rn)$-topology.
	\end{theorem}

As a variant of Theorem~\ref{gammy},
we also discuss the $\Gamma$-convergence theory with fixed exterior conditions
in the $L^1(\Omega)$ topology.
For this, given $\varphi \colon \Co \Omega\to \R$ we define
	\eqlab{ \label{roba}
		\mathcal X_\varphi^{q}(\Omega):= \left\{ u \in L^1(\Omega) \; \bigg| \; \iint_{Q(\Omega)} \frac{ |u(x)-u(y)|^{{q}}}{|x-y|^{n+s{q}}} dx \, dy <\infty, \; \; u=\varphi \mbox{ in } \Co \Omega \right\}.
		} 
	We introduce the (extended) functionals on $L^1(\Omega)$ given by
	\sys[\tilde\En_{s,\varphi}^{q} (u):=]{
		&\frac{1}{2q} \iint_{Q(\Omega)} \frac{ |u(x)-u(y)|^{{q}}}{|x-y|^{n+s{q}}} dx \, dy  && \mbox{ if } \; u \in \mathcal X_\varphi^{q} (\Omega),
		\\
		& +\infty &&\mbox{ if } \; u \in L^1(\Omega)\setminus \mathcal X_\varphi^{q} (\Omega).
		} 

		We recall that if  $q\in[1,1/s)$ and $u\in 	\mathcal X_\varphi^{q}(\Omega)$, then according to Lemma~\ref{domain_lem}, $u\in W^{s,q}(\Omega)$ and $\Tss^q (u,\Co \Omega; \cdot) \in L^1(\Omega)$. In this setting, we have the following result:
		
	\begin{theorem}	\label{fixeddata}
Let $\varphi \colon \Co \Omega \to \R$ be such that
	\eqlab{ \label{pup1}
	\limsup_{p\searrow 1 } \|\Tss^{p} (\varphi,\Co \Omega; \cdot)  
\|_{L^1(\Omega)} < \infty.
	}
	Then
		\[
			\Gamma-\lim_{p \searrow 1} \tilde \En_{s,\varphi}^{p}=\tilde \En_{s,\varphi}^1,
		\]
		in the $L^1(\Omega)$-topology.
			\end{theorem}
			
			As a side observation, note that
			condition~\eqref{pup1} cannot be dropped, as detailed in the counter-example
provided in Remark~\ref{fgr}.

Concerning assumption~\eqref{pup1}, it is also interesting to point out
	several equivalent formulations of such a hypothesis, according to the following observation.
\begin{lemma}\label{EQYU:FOR}
	Let $\varphi:\Co\Omega\to\R$ be a measurable function. The following are equivalent:
	\begin{enumerate}[label=(\roman*)]
		\item there exists $q>1$ such that
		\[
		\sup_{p\in(1,q)} \Tss^p(\varphi,\Co \Omega; \cdot) \in L^1(\Omega).
		\]
		\item There exists $q>1$ such that
		\[
		\sup_{p\in(1,q)} \|\Tss^p(\varphi,\Co \Omega; \cdot)\|_{L^1(\Omega)}<\infty.
		\]
		\item There exists $q>1$ such that $\Tss^1(\varphi,\Co \Omega; \cdot) \in L^1(\Omega)$ and $\Tss^q(\varphi,\Co \Omega; \cdot) \in L^1(\Omega)$.
	\end{enumerate}
\end{lemma}

In the following remark, we give some explanation on the meaning of the condition of integrable ``nonlocal tail'', that we consider throughout the paper. 
\begin{remark}Given $\varphi \colon \Co \Omega \to \R$, extending  it to zero inside $\Omega$ by letting $\tilde \varphi \colon \Rn \to \R$ be such that $\tilde \varphi=\varphi$ in $\Co \Omega$ and $\tilde \varphi=0$ in $\Omega$, we observe that
\eqlab{ \label{proc1} \E^q(\tilde \varphi)= q \| \Tss^q (\varphi,\Co \Omega; \cdot)\|_{L^1(\Omega)} .
	}
Hence, requiring $\varphi$ to have integrable ``nonlocal tail'' ensures the existence of a competitor with finite nonlocal energy. More precisely, the competitor having finite energy \emph{must be} $\tilde \varphi$. In case $sq>1$, such a condition is rather strong and unnatural, since it forces $\varphi$ to be close to zero near the boundary of $\Omega$ from outside, thus unnecessarily restricting the class of admissible exterior data. For example, if $\varphi= c\neq 0$ in $\Co \Omega$ and $sq>1$, then $ \E^q(\tilde \varphi)=\infty$; nonetheless the function $u=c$ in $\Rn$ is clearly a competitor with finite $\E^q$ energy. 
\\
On the other hand, for $sq<1$ (which is our working hypothesis in this paper),  in light of  Lemma~\ref{domain_lem},  the existence of \emph{an arbitrary} competitor with finite $\E^q$ energy implies that the nonlocal tail is integrable, hence the two conditions are equivalent.
\\
In particular, if $sq<1$,  then, recalling the notations in \eqref{dgggg} and  \eqref{roba},
	\[ \mathcal X_\varphi^{q}(\Omega) \neq \varnothing \qquad \mbox{if and only if} \quad 
\Tss^q(\varphi, \Omega; \cdot) \in L^1(\Omega), \]
and in this case
	$
	\mathcal X_\varphi^{q}(\Omega) =\W_\varphi^{s,q}(\Omega)$ .
\end{remark}

We also recall that the recent literature has presented
some $\Gamma$-convergence results related
to the energy $\E^q$. Namely, in \cite{brasco}, the authors study the $\Gamma$-convergence of $\E^q(u)$ for $s\nearrow 1$ and with zero exterior data, while in \cite{li2019} the $\Gamma$-limit, as $q\to \infty$, of an energy related to ours is studied (precisely, their energy consists only of the  first term in \eqref{split}, thus of only those interactions occurring inside~$\Omega$, and a boundary condition is given).
For similar results which consider only contributions in $\Omega$,
concerning more general kind of functionals, see also \cite{BN16b}.

Concerning the limit procedure as~$p\searrow1$,
we also mention the paper~\cite{MR1897460}, where some
$L^1$-inequalities are obtained by passing to the limit in~$L^p$-inequalities as $p \searrow 1$.
	\medskip
	
	We come now to point~(3)
and to the discussion of the limit Euler-Lagrange equation. This direction of research
is inspired by the notion of $(s,1)$-Laplacian  introduced in~\cite{toled}, in the spirit of the classical equation for functions of least gradient \cite{mazrossleon, Mazon}. 
More precisely, in \cite{toled}, considering the Dirichlet problem with zero boundary condition, the limit as $p\searrow 1$ of weak solutions of the $(s,p)$-problems is proved to be weak solution of the $(s,1)$- problem (recall Definitions~\ref{D:p}
and~\ref{s1lapdel}). In this paper, we adapt the approach in \cite{toled}
and we consider the limit as $p\searrow 1$
of the problem~\eqref{Dirichlet1}, whose formal
limit consists in the fractional $1$-Laplacian problem given by \eqref{Dirichlet1ss}.
 
What is more, it is known that under a suitable integrability condition on $\varphi$, weak solutions of~\eqref{Dirichlet1} are equivalent to minimizers of $\E^p$ in $\W^{s,p}_\varphi(\Omega)$
(see e.g.~\cite{dkplocal}; for our setting,
we discuss this in detail in Proposition~\ref{deellos}).  We prove here the equivalence between weak solutions of \eqref{Dirichlet1ss} and minimizers of
the energy $\E^1$ under the rather strong assumption \eqref{new}
on the exterior data $\varphi$. More precisely: 
	
\begin{theorem}\label{answer} Let  $u  \in \mathcal{W}^{s,1}(\Omega)$. The following holds:
  \begin{itemize}
  \item[(i)] If  $u$ is a weak solution  to the problem \eqref{armoni}, then $u$ is an $s$-minimal function in $\Omega$.

  \item[(ii)] Assume that there exists a weak solution $\overline{u}\in\W^{s,1}(\Omega)$ of \eqref{armoni}. Then any $s$-minimal function $u$  in $\Omega$ such that $\overline{u}=u$ almost everywhere in $\Co\Omega$, is also $u$ a weak solution of \eqref{armoni}.
 
   \item[(iii)]
   Let
\[ q\in \left(1, \min\left\{ \frac{n}{n-s}, \frac{n}{n+s-1}\right\}\right),\]
and let 
\[ s_q:=s+n-\frac{n}q \in (s,1).\]
Let $\varphi \colon \Co \Omega \to \R$ be such that
\eqlab{ \label{new}
\Ts_s^1 (\varphi,\Co \Omega; \cdot) \in L^1(\Omega), \qquad \Ts_{s_q}^q
(\varphi,\Co \Omega; \cdot) \in L^1(\Omega).}
Then, there exists a weak solution $u\in \W^{s,1}_\varphi(\Omega)$ to
the problem \eqref{Dirichlet1ss}. 
  \end{itemize}

\end{theorem}

It is interesting to observe that there is a ``mismatch'' between the conditions on $u|_{\Co \Omega}$ in points (i) and (iii) in
Theorem~\ref{answer}. On the one hand, no requirement---beside measurability---is needed in  (i) in order to ensure that a weak solution is an $s$-minimal function. On the other hand, in (iii) $\varphi$ is required to satisfy some appropriate uniform weighted integrability condition, to prove that weak solutions exist. Once a weak solution is known to exist, by point (ii), any $s$-minimal function is still a weak solution. 

The reason for the strong assumption \eqref{new} resides in the
asymptotic technique that we employ in the proof of the theorem, which
follows the argument in \cite{Mazon} (basically, the H\"{o}lder inequality is employed,
and such an $s_q$ appears to be the correct exponent, see page \pageref{necsp}).
It would be interesting to understand whether an $s$-minimal function is necessarily
a weak solution, eventually under only the natural assumption of the integrability of  $\Ts_s^1$.

	\medskip
	
It is interesting to observe that if $\|\varphi\|_{L^\infty(\Co\Omega)}<\infty$, then condition \eqref{new} is satisfied.
In particular, in the context of characteristic functions, 
we obtain the forthcoming
Corollary~\ref{coroagg}. To state this result, 
we recall that
the $s$-perimeter of a measurable set $E\subset \Rn$ in  an open set $\Omega \subset \Rn$ is given by
	\[
		\Per_s(E,\Omega)= \frac12 \iint_{Q( \Omega)} \frac{|\chi_E(x)-\chi_E(y)|}{|x-y|^{n+s}} dx \, dy
		=\E^1(\chi_E).
	 \] 
We refer to \cite{nms}, where this operator was first introduced. We recall that a set $E$ is said to be $s$-minimal in $\Omega$ if ${\rm Per}_s(E, \Omega) < \infty$ and
\begin{equation}\label{poiuytrdfghfjjk}
{\rm Per}_s(E, \Omega) \leq {\rm Per}_s(F, \Omega) \quad \hbox{for any} \ F \subset \Rn \ \ \hbox{such that} \ \ F\setminus \Omega = E\setminus \Omega.\end{equation}
In this framework, we state the following result.

\begin{corollary}\label{coroagg}
	Let $E\subset\R^n$ be such that $\Per_s(E,\Omega)<\infty$. Then, $E$ is $s$-minimal in $\Omega$ if and only if $\chi_E$ is a weak solution of $(-\Delta)^s_1\chi_E=0$ in $\Omega$.	
\end{corollary}

We now state a regularity result for the $s$-minimal functions as defined in \eqref{DEF-2}, which can be derived from the uniform density and perimeter estimates of the $s$-minimal sets.
Specifically, in the following result,
we will consider a subdomain~$\Omega'$ and obtain oscillations
and $BV$ estimates. In further detail,
we will control the supremum of $s$-minimal functions
in terms of the mass of their positive part in the domain and the distance
between the given subdomain and the boundary of the original domain
(similarly, one can
control the infimum
in terms of the mass of the negative part and the domain distance).
Furthermore, we bound the $BV$ seminorm
of the $s$-minimal function by the mass in the domain.
Our precise result goes as follows:

\begin{theorem}\label{thmreg}
	If $u\in\W^{s,1}(\Omega)$ is $s$-minimal in $\Omega$, then $u\in L^\infty_{loc}(\Omega)\cap BV_{loc}(\Omega)$. More precisely, for every $\Omega'\Subset\Omega$,
	\eqlab{\label{linftysupest}
	\sup_{\Omega'}u\leq \frac{1}{c\dist(\Omega',\partial\Omega)^n}\|u_+\|_{L^1(\Omega)}\quad\mbox{and}\quad\inf_{\Omega'}u\geq -\frac{1}{c\dist(\Omega',\partial\Omega)^n}\|u_-\|_{L^1(\Omega)},
}
	where $u_+:=\max\{0,u\}$ and $u_-:=\min\{0,u\}$, and $c=c(n,s)>0$ is the constant of the uniform density estimates in \cite[Theorem 4.1]{nms}, and there exists a positive constant $C$ depending only on $n,s,\Omega'$ and $\dist(\Omega',\partial\Omega)$, such that
	\eqlab{\label{bvlocest}
	|Du|(\Omega')\leq C\|u\|_{L^1(\Omega)}.	
}
\end{theorem}

We conclude this introduction with some remarks.
We notice that we can rephrase the definition
in~\eqref{poiuytrdfghfjjk}
by saying that a set $E$ is $s$-minimal in $\Omega$ if and only if $\chi_E$ minimizes $\E^1$ within the subspace of $\mathcal{W}^{s,1}(\Omega)$ consisting of those characteristic functions which are equal to $\chi_E$ outside $\Omega$.

Actually, by making use of an appropriate co-area formula, recently in \cite{bdlv20} it was proved that a set $E$ is $s$-minimal in $\Omega$ if and only if $\chi_E$ is an $(s,1)$-minimizer in $\Omega$---that is, $\chi_E$ minimizes $\E^1$ not only among characteristic functions, but among \emph{arbitrary} functions belonging to $\mathcal{W}^{s,1}_{\chi_E}(\Omega)$.

Furthermore, the connection (in the classical framework) between functions of least gradient and functions with area minimizing level sets was studied in the seminal paper by Bombieri, De Giorgi and Giusti \cite{bomby} and later by Sternberg, Williams and Ziemer \cite{SWZ}. Recently, Bucur, Dipierro, Lombardini and Valdinoci \cite{bdlv20} have obtained the nonlocal counterpart, in the fractional setting, of the above classical results by establishing, on the one hand, that
if $u \in \mathcal{W}^{s,1}(\Omega)$ is an $s$-minimal function in $\Omega$, then, for all $\lambda \in \R$, the set $\{ u \geq \lambda \}$ is $s$-minimal in $\Omega$,
and, on the other hand, that
if $u \in \mathcal{W}^{s,1}(\Omega)$ and $\{ u \geq \lambda \}$ is $s$-minimal in $\Omega$ for almost every $\lambda \in \R$, then $u$ is an $s$-minimal function in $\Omega$.

\medskip

The rest of this paper is organized as follows. Section~\ref{SEF:cv}
contains the proofs of Lemmata \ref{domain_lem} and~\ref{EQYU:FOR}.
In Section~\ref{7yhbf}
we give the proof of Theorem~\ref{theorem}, while Section~\ref{GAMMASE}
is devoted to the proofs of Theorems~\ref{gammy}
and~\ref{fixeddata}, and the proof
of Theorem~\ref{answer}
is contained in Section~\ref{tygh7t798fdwevirbg}.
The final appendix \ref{appendixxx} contains additional comments on the
asymptotics as~$s\to1^-$
and the proof of Theorem~\ref{thmreg}.

\section{Proof of Lemmata \ref{domain_lem} and~\ref{EQYU:FOR}}\label{SEF:cv}

This section is devoted to the proofs of Lemmata \ref{domain_lem} and~\ref{EQYU:FOR}.
We start by considering the second of these lemmata.

\begin{proof}[Proof of Lemma~\ref{EQYU:FOR}]
It is easy to notice that
	\bgs{		\sup_{p} \|\Tss^p(u,\Co \Omega; \cdot)\|_{L^1(\Omega)} = &\;\sup_p \int_\Omega\left( \int_{\Co \Omega} \frac{|u(y)|^p}{|x-y|^{n+sp}} dy \right) dx
	\\
 \leq &\;	 \int_\Omega \left( \sup_p  \int_{\Co \Omega} \frac{|u(y)|^p}{|x-y|^{n+sp}}  dy \right) dx =  \| \sup_{p} \Tss^p(u,\Co \Omega; \cdot)\|_{L^1(\Omega)}.
 }
Thus, we only need to prove the implication $(iii)\implies(i)$. For this, notice that for every $p\in(1,q)$ there exists a unique $t_p\in(0,1)$ such that $p=t_p+(1-t_p)q$. Thus, by Young's inequality, we obtain
	\bgs{
	\int_\Omega& \left( \sup_{p\in(1,q)}  \int_{\Co \Omega} \frac{|\varphi(y)|^p}{|x-y|^{n+sp}}  dy \right) dx
=\int_\Omega \left( \sup_{t\in(0,1)}  \int_{\Co \Omega} \frac{|\varphi(y)|^{t}}{|x-y|^{(n+s)t}}\frac{|\varphi(y)|^{(1-t)q}}{|x-y|^{(1-t)(n+sq)}}  dy \right) dx\\
&
\qquad\leq\int_\Omega \left( \sup_{t\in(0,1)}  \int_{\Co \Omega} t\,\frac{|\varphi(y)|}{|x-y|^{n+s}}+(1-t)\,\frac{|\varphi(y)|^{q}}{|x-y|^{n+sq}}  dy \right) dx\\
&
\qquad\leq \|\Tss^1(\varphi,\Co \Omega; \cdot)\|_{L^1(\Omega)}+\|\Tss^q(\varphi,\Co \Omega; \cdot)\|_{L^1(\Omega)}.
}
This concludes the proof of Lemma~\ref{EQYU:FOR}.
\end{proof}

Concerning claim~(iii) in Lemma~\ref{EQYU:FOR},
we stress that in order to have an equivalent statement,
it is not enough to require that there exists $q>1$ such that $\Tss^q(\varphi,\Co \Omega; \cdot) \in L^1(\Omega)$, since the tail $\Tss^1(\varphi,\Co \Omega; \cdot)$ might not be integrable: an explicit
example of this phenomenon goes as follows.

\begin{example}\label{Bolas}
Let $\Omega\Subset B_R\subset\R^n$, with $R>2$. Consider the function
\sys[ \varphi(y) :=]{& \frac{|y|^s}{\log |y|} && \mbox{  in } \;  \Co B_R,
	\\
	&0 && \mbox{  in } \;B_R.}
Then, $\Tss^p(\varphi,\Co\Omega;\cdot)\in L^1(\Omega)$ for every $p>1$, and $\Tss^1(\varphi,\Co\Omega;\cdot)\not\in L^1(\Omega)$.

Indeed, notice at first that since $\Omega\Subset B_R$ there exist two constants $a,b>0$, depending on $\Omega$ and $R$, such that
\[
a|y|\leq|x-y|\leq b|y|\quad\mbox{for every }x\in\Omega\mbox{ and }y\in\Co B_R.
\]
Hence
\bgs{
\|\Tss^p(\varphi,\Co\Omega;\cdot)\|_{L^1(\Omega)}=\int_\Omega\left(\int_{\Co B_R}\frac{|y|^{sp}}{(\log |y|)^p|x-y|^{n+sp}}\right)dx
\leq \frac{|\Omega|}{a^{n+sp}}\int_{\Co B_R}\frac{dy}{(\log |y|)^p|y|^n}<\infty,
}
for every $p>1$.
On the other hand,
\bgs{
	\|\Tss^1(\varphi,\Co\Omega;\cdot)\|_{L^1(\Omega)}=\int_\Omega\left(\int_{\Co B_R}\frac{|y|^{s}}{(\log |y|)|x-y|^{n+s}}\right)dx
	\geq \frac{|\Omega|}{b^{n+sp}}\int_{\Co B_R}\frac{dy}{(\log |y|)|y|^n}=+\infty.
}
\end{example}

\begin{remark}\label{yup1} 
While the equivalent claims in Lemma~\ref{EQYU:FOR} may look, at a very first glance, focused on rather
fussy assumptions, we observe that they are satisfied in many simple cases of interest.
For instance, let $\varphi \in L^\infty(\Co \Omega)$, and let $q\in [1,1/s)$. Then 
	\begin{equation}\label{098765} \sup_{p\in[1,q]} \Tss^p(\varphi,\Co \Omega; \cdot)  \in L^1(\Omega).\end{equation}
    Indeed, we note that if $|x-y| \geq 1$, then $|x-y|^{n+sp} \geq |x-y|^{n+s}$ and if $|x-y| \leq 1$, then $|x-y|^{n+sp} \geq |x-y|^{n+sq}$. Exploiting this observation, we obtain
\eqlab{\label{bleah3}
	\sup_{p\in[1,q]} \Tss^{p}(\varphi,\Co \Omega; x)  \leq &\;\|\varphi\|_{L^\infty(\Omega)}
		\sup_{p\in[1,q]} \int_{\Co \Omega} \frac{ dy}{|x-y|^{n+sp}} 
		\\
		\leq &\; \|\varphi\|_{L^\infty(\Omega)}\left(\int_{\Co \Omega \cap \{ |x-y| \geq 1\} } \frac{ dy}{|x-y|^{n+s}}+  \int_{\Co \Omega \cap \{ |x-y| < 1\} } \frac{ dy}{|x-y|^{n+sq}} \right)\\
		\leq &\; \|\varphi\|_{L^\infty(\Omega)}\left(\int_{\Co \Omega } \frac{ dy}{|x-y|^{n+s}}+  \int_{\Co \Omega } \frac{ dy}{|x-y|^{n+sq}} \right),
		}
	for every $x\in\Omega$. We recall that, since $\Omega$ is bounded and has Lipschitz boundary, 
		\[	
			 \int_{\Omega} \int_{\Co \Omega} \frac{dx\, dy}{|x-y|^{n+\sigma}} =\Per_{\sigma}(\Omega,\Rn)<\infty,
		\]
	for every $\sigma\in(0,1)$. Thus, by \eqref{bleah3}, we conclude that
	\[
	\int_\Omega\sup_{p\in[1,q]} \Tss^p(\varphi,\Co \Omega; x)\,dx
	\le  \|\varphi\|_{L^\infty(\Omega)}\left(\Per_{s}(\Omega,\Rn)+\Per_{sq}(\Omega,\Rn)\right)<\infty,
	\]
	as claimed in~\eqref{098765}.
	\end{remark}
		
We can further weaken the global boundedness condition on $\varphi$ given in Remark~\ref{yup1},
providing a larger class of functions which satisfy the equivalent conditions in Lemma~\ref{EQYU:FOR},
according to the following result:

	\begin{lemma} \label{yup3} Let $q\in \left(1,1/s\right)$ and let $\varphi \in W^{s,q}(\Co \Omega)$. Then 
	\eqlab{ \label{yup2} \sup_{p\in [1,q]} \Tss^{p}(\varphi,\Co \Omega; \cdot)  \in L^1(\Omega).}
	\end{lemma}
	
	\begin{remark}
	The statement of Lemma~\ref{yup3} can be sharpened, by relaxing the assumption
	that~$\varphi \in W^{s,q}(\Co \Omega)$. For any $\eta>0$, we use the notation
\eqlab { \label{ometa}\Omega_\eta:= \{ y\in \R^n \; | \; \dist(y,\Omega)<\eta\}.}
As a matter of fact, given~$\eta>0$, one can consider the following alternative conditions:
\begin{multicols}{3}
\begin{enumerate}[label=(\alph*)]
\item  either $\varphi \in L^\infty( \Omega_\eta\setminus \Omega)$,
 \item or 
 $ \varphi \in W^{s,q}( \Omega_\eta\setminus \Omega)$, 
 \end{enumerate} 
 ${ } \qquad\qquad $ and 
 \begin{enumerate}[label=(\Alph*)]
\item  either $\varphi \in L^\infty( \Co \Omega_\eta)$,
 \item or $\varphi \in L^q(\Co \Omega_\eta)$.
 \end{enumerate} 
\end{multicols}
Indeed, with computations which are similar to those of the proofs of
Lemma~\ref{yup3} and Remark~\ref{yup1}, one notices that~\eqref{yup2} holds for any combination $(i)\&(I)$, with $i\in\{a,b\}$, $I\in\{A,B\}$.

On the other hand, we stress that assuming that~$\Tss^{q}(\varphi,\Co \Omega; \cdot)  \in L^1(\Omega)$ is not enough to ensure the validity of \eqref{yup2}---see Example \ref{Bolas}.
	\end{remark} 

\begin{proof}[Proof of Lemma~\ref{yup3}]
We notice that for $|x-y|<1$
	\[
		\frac{|\varphi(y)|^p}{|x-y|^{n+sp}} \leq \frac{|\varphi(y)|^q+1}{|x-y|^{n+sq}} 
,\]
while for $|x-y|\geq 1$ we have that
\[ \frac{|\varphi(y)|^p}{|x-y|^{n+sp}}\leq |\varphi (y)|^q + \frac{1}{|x-y|^{n+s}}.\]
Then we obtain
	\bgs{
 \|\Tss^p(\varphi,\Co \Omega; x) \|_{L^1(\Omega)}\leq&\;  \int_{\Omega} \int_{\Co \Omega \cap\{ |x-y|<1\} } \frac{ |\varphi(y)|^q}{|x-y|^{n+sq}}dx \, dy  + \int_{\Omega} \int_{\Co \Omega} \frac{ dx \, dy}{|x-y|^{n+sq}} 
\\
&\;
+  \int_{\Omega} \int_{\Co \Omega \cap\{ |x-y|\geq 1\} }  |\varphi(y)|^q\, dx \, dy +\int_{\Omega} \int_{\Co \Omega} \frac{ dx \, dy}{|x-y|^{n+s}} \\
	\leq &\; 
 \int_{\Omega} \int_{B_R\setminus \Omega} \frac{ |\varphi(y)|^q}{|x-y|^{n+sq}}dx \, dy +  |\Omega| \|\varphi\|_{L^q(\Co \Omega)}^q + \Per_{sq}(\Omega,\Rn) +\Per_{s}(\Omega,\Rn),
}	
for some $R>0$ large enough.
Using Fubini-Tonelli's Theorem and \eqref{fracH}, we have that
\begin{equation}\begin{split}\label{BBB} 
	 		\int_{\Omega} \left( \int_{B_R \setminus \Omega}\frac{ |\varphi(y)|^q  }{|x-y|^{n+s q}}dy \right) dx 
	 		\leq  &\; 
	 		\int_{B_R\setminus \Omega} |\varphi(y)|^q \left( \int_{\Co \left(B_R \setminus \Omega\right) } \frac{dx}{|x-y|^{n+s q}} \right) dy  
\\
	 		\leq &\; C(n,s,q,\Omega) \|\varphi\|_{W^{s,q}(B_R \setminus \Omega)}^q.
	 		\end{split}\end{equation}
Moreover, for every~$\sigma\in\{s,sq\}\subset (0,1)$, we have that 
		\bgs{
			 \int_{\Omega} \int_{\Co \Omega} \frac{dx\, dy}{|x-y|^{n+\sigma}} =\Per_{\sigma}(\Omega,\Rn)<\infty.
		}
With this, we reach the conclusion of the lemma.\end{proof}

	We point out that  the problem of looking for a minimizer of $\E^q$, for $q$ close enough to $1$, in the space $\W^{s,q}(\Omega)$, is well posed if and only if the tail is summable, as stated in Lemma~\ref{domain_lem}, that we now prove.

\begin{proof}[Proof of Lemma~\ref{domain_lem}]
	First of all we observe that, since $sq<1$ and $\Omega$ is a bounded open set with Lipschitz boundary, by the fractional Hardy-type inequality
(stated without a proof  in \cite[(17)]{Dyda}), we have that for all $v\in W^{s,q}(\Omega)$ it holds
	\eqlab{ \label{fracH}
		\int_\Omega \left(\int_{\Co \Omega} \frac{ |v(x)|^{q}}{|x-y|^{n+sq}} dy\right) dx \leq  \int_\Omega \frac{|v(x)|^q} {(\mbox{dist}( x, \partial \Omega))^{sq}} dx \leq \overline C \|v\|^q_{W^{s,q}(\Omega),}
	}
for some constant $\overline C:=\overline C(n,s,q,\Omega)>0$.
For a simple proof of the fractional Hardy inequality, based on the fractional Hardy
inequality on half-spaces of \cite{Seir}, see also \cite[Proposition A.2]{teolu}.

	Suppose now that $\E^q(u)<\infty$. Then, from \eqref{split},
	\[
	[u]_{W^{s,q}(\Omega)}^q\le2 q\E^q(u)<\infty.
	\]
Interestingly, this is enough to conclude that $\|u\|_{L^q(\OMega)}<\infty$, without having to make any kind of integrability assumption on $u$. Indeed, since $u$ is measurable and finite almost everywhere, we can find a measurable set $E\subseteq\Omega$ such that $|E|>0$ and $\int_E|u|dx<+\infty$. Then, a simple computation shows that
	\[
	\|u\|_{L^q(\Omega)}^q\leq\frac{2^{q-1}}{|E|}\left\{\mbox{diam}(\Omega)^{n+sq}[u]_{W^{s,q}(\Omega)}^q+\frac{|\Omega|}{|E|^{q-1}}\Big|\int_E u(\xi)\,d\xi\Big|^q\right\}<+\infty.
	\]
	For the details we refer to \cite[Lemma D.1.2]{tesilu} and its proof. Thus, if $u$ is a measurable function such that $\E^q(u)<\infty$, we have that~$u\in W^{s,q}(\Omega)$.

	Now, if we denote by $\bar u:\R^n\to\R$ the function
	\sys[ \bar u:=]{ & u && \mbox { in } \; \Omega,
	\\ & 0 && \mbox{ in  } \; \Co \Omega,} 
by \eqref{fracH} we have
	 \[
	\E^q(\bar u)=\frac{1}{2q} [u]^q_{W^{s,q}(\Omega)}
	+\frac{1}{q}\int_\Omega \left(\int_{\Co \Omega} \frac{ |u(x)|^{q}}{|x-y|^{n+sq}} dy\right) dx<\infty.
	\]
	Therefore we obtain
	\[
	\|\Tss^q (u, \Co \Omega; x)\|_{L^1(\Omega)}=q\, \E^q(u-\bar u)\le q\, 2^{q-1}(\E^q(u)+\E^q(\bar u))<\infty.
	\]
	 
	For the converse implication we can argue similarly. Since $u\in W^{s,q}(\Omega)$ by hypothesis, exploiting once again \eqref{fracH} we have $\E^q(\bar u)<\infty$. We define also
	\sys[ u_0:=]{ & 0 && \mbox { in } \; \Omega,
		\\ & u && \mbox{ in  } \; \Co \Omega,}
	and we observe that
	\[
	\E^q(u_0)=\frac{1}{q}\|\Tss^q (u, \Co \Omega; x)\|_{L^1(\Omega)}<\infty.
	\]
As a consequence,
	\[
	\E^q(u)=\E^q(u_0+\bar{u})\le 2^{q-1}(\E^q(u_0)+\E^q(\bar u))<\infty.
	\]
	This concludes the proof of the desired result.
	 \end{proof}
	 
\section{Convergence of minimizers of the $(s,p)$-energy to minimizers of the $(s,1)$-energy}\label{7yhbf}

The main goal of this section is to prove Theorem~\ref{theorem}. For this,
we give some auxiliary technical results of general flavor.

\subsection{A continuous embedding}

As a technical tool, we prove the continuous embedding  $W^{s,p}(\Omega)\hookrightarrow W^{\sigma,1}(\Omega)$ for any $p\in [1,\infty)$, with $\sigma<s$. 
For this, we start with an auxiliary
inequality. 

\begin{lemma} \label{lemmatwo} Let $p\in(1,+\infty)$ and $\sigma\in(0,s)$. It holds that
	\eqlab{ \label{tre}
		 \; [u]_{W^{\sigma,1}(\Omega)} \leq \left( \frac{C(n,\Omega)(p-1)}{p(s-\sigma)}\right)^{\frac{p-1}p} C(\Omega)^{s-\sigma} [u]_{W^{s,p}(\Omega)}, 	
	}
	where
	\[
	C(n,\Omega):=|\Omega|\Ha^{n-1}(\partial B_1)\quad\mbox{and}\quad
	C(\Omega):=\diam(\Omega).
	\]
	\end{lemma} 
	
	\begin{proof}
	Notice that if $[u]_{W^{s,p}(\Omega)} =\infty$, there is nothing to prove. Else, if  $[u]_{W^{s,p}(\Omega)} <\infty$, using H\"{o}lder's inequality, we get that
	\eqlab{ \label{uno}
		\; [u]_{W^{\sigma,1}(\Omega)} =&\; 
		 \int_\Omega \int_\Omega \frac{|u(x)-u(y)|}{|x-y|^{ \frac{n}p +s}} \frac{dx \, dy}{|x-y|^{ \frac{n(p-1)}p + \sigma -s}} 
	\\
	 \leq &\; \left(\int_\Omega \int_\Omega \frac{|u(x)-u(y)|^p}{|x-y|^{ n +sp}} dx \, dy\right)^{\frac1p} 
	\left(\int_\Omega \int_\Omega \frac{dx \, dy}{|x-y|^{ n + \frac{p(\sigma -s)}{p-1} } }\right)^{\frac{p-1}{p}}
	\\ 
	= &\; [u]_{W^{s,p}(\Omega)}\;  \mathcal J(p)^{\frac{p-1}p}.
	} 
	Changing variables $z=x-y$, using the notation $d:= \diam(\Omega)$ and recalling that $s-\sigma>0$ and $p>1$, we get 
	\bgs{
		\mathcal J(p) \leq &\; \int_{\Omega} \bigg( \int_{B_{d}(x)} \frac{dy}{|x-y|^{n+\frac{p(\sigma-s)}{p-1} }} \bigg) dx=\int_{\Omega} \bigg( \int_{B_{d}} \frac{dz}{|z|^{n+\frac{p(\sigma-s)}{p-1} }} \bigg) dx
	\\
	= &\;|\Omega|\Ha^{n-1}(\partial B_1) \frac{(p-1)d^{ \frac{p(s-\sigma)}{p-1} }}{ p(s-\sigma)}.
	}
We raise to the power~$\frac{p-1}p$ to obtain that
	\bgs{
		J(p)^{\frac{p-1}p}  \leq\left( \frac{C(n,\Omega)(p-1)}{p(s-\sigma)}\right)^{\frac{p-1}p}d^{s-\sigma}. 
	}
	From this and \eqref{uno} we obtain the claim in \eqref{tre}.
\end{proof}

As a consequence of Lemma~\ref{lemmatwo}, we obtain the following embedding result. 	
	 \begin{theorem}\label{quattro} Let $\sigma\in(0,s)$. Then the continuous embedding $W^{s,p}(\Omega)\hookrightarrow W^{\sigma,1}(\Omega)$ holds for every $p\in[1,\infty)$. In particular, for any $u\in W^{s,p}(\Omega)$,
	\bgs{ 
	 \; \|u\|_{W^{\sigma,1}(\Omega)} \leq \mathfrak C_1 \|u\|_{W^{s,p}(\Omega)} ,
	}
	where $\mathfrak C_1:= \mathfrak C_1(n,s,\sigma,\Omega)>0$ does not depend on $p$.
\end{theorem}

\begin{proof}
	First of all, using the notation of Lemma~\ref{lemmatwo}, we observe that
	\[
	C_1(n,s,\sigma,\Omega):=\sup_{p\in(1,\infty)}\left( \frac{C(n,\Omega)(p-1)}{p(s-\sigma)}\right)^{\frac{p-1}p}<\infty.
	\]
	Therefore, by Lemma~\ref{lemmatwo}, we have that
	\[
	[u]_{W^{\sigma,1}(\Omega)} \leq C_1(n, s,\sigma,\Omega)C(\Omega)^{s-\sigma}  [u]_{W^{s,p}(\Omega)},
	\]
	for every $p\in(1,\infty)$. Moreover, by H\"{o}lder's inequality
	\bgs{
		\|u\|_{L^1(\Omega)} \leq |\Omega|^\frac{p-1}{p} \|u\|_{L^p(\Omega)}\leq C_2(\Omega)\|u\|_{L^p(\Omega)},
	}
where
\[
C_2(\Omega):=\sup_{p\in(1,\infty)}|\Omega|^\frac{p-1}{p}<\infty.
\]
This implies that
\eqlab{\label{tortuga666}
\|u\|_{W^{\sigma,1}(\Omega)} \leq \mathfrak C_1 \|u\|_{W^{s,p}(\Omega)},
}
for every $p\in(1,\infty)$, with
	\[
	\mathfrak C_1= C(n,s,\sigma,\Omega):=C_1(n, s,\sigma,\Omega)C(\Omega)^{s-\sigma}+C_2(\Omega).
	\]
	
	We are left to prove the claim in the case $p=1$. In order to do this, we exploit an approximation argument. By the density of $C^\infty_c(\Omega)$ in $W^{s,1}(\Omega)$, 
	given $u\in W^{s,1}(\Omega)$ we can find a sequence $\{u_k\}_k\subset C^\infty_c(\Omega)$ such that
	\eqlab{\label{killer_quokka}
	\lim_{k\to\infty}\|u-u_k\|_{W^{s,1}(\Omega)}=0.
}
	Notice that, since $\mathfrak C_1$ does not depend on $p$ and $u_k\in C^\infty_c(\Omega)$, by Lebesgue's Dominated Convergence Theorem we have
\bgs{
	\lim_{p\searrow1}\mathfrak C_1\|u_k\|_{W^{s,p}(\Omega)}=\mathfrak C_1\|u_k\|_{W^{s,1}(\Omega)}.
}
	Hence, passing to the limit as~$p\searrow1$ in \eqref{tortuga666}, we obtain
	\[
	\|u_k\|_{W^{\sigma,1}(\Omega)} \leq \mathfrak C_1 \|u_k\|_{W^{s,1}(\Omega)},
	\]
	for every $k\in\mathbb N$. Then, passing to the limit as~$k\to\infty$, exploiting Fatou's Lemma and \eqref{killer_quokka}, we find
	\[
	\|u\|_{W^{\sigma,1}(\Omega)}\le\liminf_{k\to\infty}\|u_k\|_{W^{\sigma,1}(\Omega)}\le\mathfrak C_1 \lim_{k\to\infty}\|u_k\|_{W^{s,1}(\Omega)}=\mathfrak C_1 \|u\|_{W^{s,1}(\Omega)},
	\]
	concluding the proof of the desired result.
\end{proof}

\subsection{A priori estimates}
Now we state
a bound of the $\|\cdot\|_{W^{s,p}(\Omega)}$-norm for
$(s,p)$-minimizers
that is uniform with respect to $p$. 

\begin{lemma} \label{lemmaone} Let $p\ge1$, and let $\varphi \colon \Co \Omega \to \R$ such that $\Tss^p(\varphi, \Co \Omega; \cdot)\in L^1(\Omega)$.
If $u_p\in \W^{s,p}_\varphi(\Omega)$ is an $(s,p)$-minimizer, then
\eqlab{\label{never_enough_remarks}
	 \; \|u_p\|_{W^{s,p}(\Omega)} \leq \mathfrak C_0\|\Tss^p(\varphi, \Co \Omega; \cdot)\|_{L^1(\Omega)}^\frac{1}{p}\leq \mathfrak C_0\left(1+\|\Tss^p(\varphi, \Co \Omega; \cdot)\|_{L^1(\Omega)}\right),
}
	with $\mathfrak C_0:= C(n,\Omega)>0$, independent of $p$ and $s$.
\end{lemma}

\begin{proof}
The strategy of the proof is to first control the Gagliardo
seminorm of~$u_p$ and then to bound its~$L^p$-norm
by using \eqref{split} and \eqref{enbdd0}. The details
go as follows.

	As a competitor for $u_p$ we define $v\in \W^{s,p}_\varphi(\Omega)$  
	\sys[v:=]{
		& 0 && \mbox{ in } \; \Omega,
		\\
		& \varphi && \mbox{ in } \; \Co \Omega.
		}	
	Exploiting the minimality and comparing the energies, we have  that
		\eqlab{ \label{enbdd0}
		\E^p(u_p) \leq &\;  \E^p(v) = \frac1p \int_\Omega \int_{\Co \Omega} 
		\frac{|\varphi(y)|^p}{|x-y|^{n+sp}} \, dx \, dy 
		=\frac1p \|\Tss^p(\varphi,\Co \Omega; \cdot)\|_{L^1(\Omega)}.
		}
	Whereas using the expression of the energy in \eqref{split}, we get that
	\eqlab{\label{normpbdd}
	 [u_p]_{W^{s,p}(\Omega)} \leq 2^{\frac1p} \|\Tss^p(\varphi,\Co \Omega; \cdot)\|_{L^1(\Omega)}^\frac{1}{p}.	
	}
 
Denoting $d:= \diam(\Omega)$, recalling \eqref{ometa} and noticing that 
	for $
	x\in \Omega,$ and $ y\in \Omega_d\setminus \Omega,
	$
	we have that
	$  |x-y|\leq 2d,
	$ 
	we obtain
	\eqlab{ \label{poin}
		\|u_p\|_{L^p(\Omega)}^p =&\; \int_{\Omega} |u_p(x)-\varphi(y)+\varphi(y)|^p \, dx 
		\\
		 = &\; \frac{1}{|\Omega_d\setminus \Omega|}\int_\Omega \left( \int_{\Omega_d \setminus \Omega} 
		|u_p(x)-\varphi(y)+\varphi(y)|^p\, dy \right) dx
		\\
		 \leq &\;\frac{2^{p-1}}{|\Omega_d\setminus \Omega|}
				 \int_\Omega \left( \int_{\Omega_d \setminus \Omega}\frac{ |u_p(x)-\varphi(y)|^p+|\varphi(y)|^p}{|x-y|^{n+sp}}|x-y|^{n+sp}\,  dy \right) dx
				 \\
				  \leq &\;\frac{2^{p-1} (2d)^{n+sp}}{|\Omega_d\setminus \Omega|} \left[\int_\Omega \left( \int_{\Omega_d \setminus \Omega}\frac{ |u_p(x)-\varphi(y)|^p}{|x-y|^{n+sp}}\,  dy \right) dx + \int_\Omega\left( \int_{\Omega_d \setminus \Omega} \frac{ |\varphi(y)|^p}{|x-y|^{n+sp}}\,  dy \right) dx \right] 
			 \\
		 	 \leq &\;\frac{2^{p-1} (2d)^{n+sp}}{|\Omega_d\setminus \Omega|}\left[p\,\E^p(u_p)+ \|\Tss^p(\varphi,\Co \Omega; \cdot) \|_{L^1(\Omega)}\right]\\
		 	 \leq &\;\frac{2^p (2d)^{n+sp}}{|\Omega_d\setminus \Omega|}\|\Tss^p(\varphi,\Co \Omega; \cdot) \|_{L^1(\Omega)},
		 	 		 	 } 
	 according to \eqref{split} and \eqref{enbdd0}.
	 We  deduce from \eqref{poin} that
	 \eqlab{\label{kangaroo_stardust}
	 \|u_p\|_{L^p(\Omega)}\leq C(n,\Omega)\|\Tss^p(\varphi,\Co \Omega; \cdot) \|_{L^1(\Omega)}^\frac{1}{p},
	}
	 where
	 \[
	 C(n,\Omega):=2\sup_{s\in(0,1)}(2d)^s\sup_{p\in[1,\infty)}\left(\frac{(2d)^n}{|\Omega_d\setminus \Omega|}\right)^\frac{1}{p}.
	 \]
	 Putting together \eqref{normpbdd} and \eqref{kangaroo_stardust} we conclude the proof of the desired result.
\end{proof}

\subsection{Convergence results}\label{us}

It is convenient to point out the following semicontinuity property, which is a consequence of Fatou's Lemma.
 \begin{lemma}\label{liminf}
 Let $p_k \searrow 1$ as $k\to \infty$ and let $u_k,u_1\colon \Rn \to \R$ be such that
 \[
 	 u_k \xrightarrow[k\to \infty]{} u_1 \qquad \mbox{ a.e. in } \quad \Rn.
 \]
 Then
 	\[
 		\E^1(u_1)\leq \liminf_{k\to \infty} \E^{p_k}(u_k).
 	\]
 \end{lemma}

In the next result we prove that  any function in the space $\W^{s,1}(\Omega)$ can be approximated, in an appropriate sense, by smooth functions. Precisely: 
\begin{lemma}\label{lemmasof3}
Let $v\in \W^{s,1}(\Omega)$ such that $\Tss^1(v,\Co\Omega;\cdot)\in L^1(\Omega)$. Then, there exists
a sequence of functions $\psi_j \colon \Rn \to \R$ such that $\psi_j|_{\Omega}\in C^\infty_c(\Omega)$, $\psi_j=v $ in $\Co \Omega$, 
\bgs{
 	 \lim_{j\to \infty} \| \psi_j -v\|_{W^{s,1}(\Omega)} = 0
 	 }
 	 and
  \bgs{
 \lim_{j\to \infty} \E^1(\psi_j) = \E^1(v) .}
\end{lemma}
\begin{proof}
 By density of $C_c^\infty(\Omega)$ in $W^{s,1}(\Omega)$, 
 there exists a sequence $(\psi_j)_j\subset C^\infty_c(\Omega)$, satisfying 				\bgs{
 	 \| \psi_j -v\|_{W^{s,1}(\Omega)} \xrightarrow[j\to \infty]{} 0.
 	 }
  We extend the functions $\psi_j$ to the whole of $\R^n$ by setting  $\psi_j:=v $ in $\Co \Omega$.
 By the triangle inequality,
			\bgs{
			 &	\left|	\iint_{Q(\Omega)} \frac{|\psi_j(x)-\psi_j(y)|}{|x-y|^{n+s}} dx \, dy - 	\iint_{Q(\Omega)}  \frac{|v(x)-v(y)|}{|x-y|^{n+s}} dx \, dy\right|
		\\   & \qquad \qquad 	\leq 
						\iint_{Q(\Omega)}   \frac{|(\psi_j-v)(x)-(\psi_j-v)(y) |}{|x-y|^{n+s}} dx \, dy.
						}
 Now
 \[   [ \psi_j -v]_{W^{s,1}(\Omega)}
 	= \int_\Omega \int_\Omega   \frac{|(\psi_j-v)(x)-(\psi_j-v)(y) |}{|x-y|^{n+s}} dx \, dy
 	,
 \] 
 hence
			\eqlab{ \label{sette1}
				\lim_{j\to \infty} \int_\Omega \int_\Omega \frac{|\psi_j(x)-\psi_j(y)|}{|x-y|^{n+s}} dx \, dy =  \int_\Omega \int_\Omega \frac{|v(x)-v(y)|}{|x-y|^{n+s}} dx \, dy
			.}
			On the other hand, since $\psi_j=v$ on $\Co \Omega$, as a consequence of \eqref{fracH}, it holds that
			\[
			\int_\Omega \int_{\Co \Omega} \frac{|(\psi_j-v)(x) - (\psi_j-v)(y)|}{|x-y|^{n+s}} =
				\int_\Omega \int_{\Co \Omega} \frac{|\psi_j(x)-v(x)|}{|x-y|^{n+s}} \leq C(n,s,\Omega) \|\psi_j-v\|_{W^{s,1}(\Omega)} \xrightarrow[j\to \infty]{} 0, 
			\] 
and therefore
		\eqlab{ \label{otto1}
		\lim_{j\to \infty}	\int_\Omega \int_{\Co \Omega}  \frac{|\psi_j(x)-\psi_j(y)|}{|x-y|^{n+s}} dx \, dy =	\int_\Omega \int_{\Co \Omega}  \frac{|v(x)-v(y)|}{|x-y|^{n+s}} dx \, dy.
		}
		Summing up \eqref{sette1} and \eqref{otto1}, we establish 
		the desired claim.
\end{proof}

In this next result, we extend the result in Lemma~\ref{lemmasof3} to the following context: we prove that the limit for $k\to \infty$ of the $(s,p_k)$-energy of a sequence $\psi_k$ of smooth functions, with suitable uniform bound  on the exterior data, is the $(s,1)$-energy of the limit function.   
\begin{lemma}\label{lemmasof1}
Let $p_k\searrow 1$ as $k\to \infty$ , let $\varphi_k \colon \Co \Omega \to \R$ be such that
\eqlab{ \label{sof2}
		\sup_{k\in \N} \Tss^{p_k}(\varphi_k,\Co\Omega; \cdot)\in L^1(\Omega)
} 
and let $\psi_k\colon \Rn \to \R$ be such that $\psi_k|_\Omega = \varphi_k$ in $\Co \Omega$, $\psi_k \in C^\infty_c(\Omega)$ and
\[
	\overline c:=\sup_{k\in \N} \|\psi_k \|_{C^1(\Omega)} < \infty.
\]
Let  also
$\psi \colon \Rn \to \R$ be such that
	\[
	\psi_k \xrightarrow[k\to \infty]{} \psi
	\qquad \mbox{ a.e. in } \quad \Rn.
	\]
	Then 
\bgs{
 	\lim_{k\to \infty} \E^{p_k} (\psi_k)= \E^1(\psi).
 	}
 \end{lemma}	
 \begin{proof}
 In the next lines, we denote by $C$ a constant, independent of $k$, that may change value from line to line. 
 	At first, we prove that
	\eqlab{\label{smop1}
		\lim_{k \to \infty} 	\frac{1}{p_k}\int_\Omega \int_\Omega \frac{|\psi_k(x)-\psi_k(y)|^{p_k}}{|x-y|^{n+sp_k}}dx \, dy =\int_\Omega \int_\Omega \frac{|\psi(x)-\psi(y)|}{|x-y|^{n+s}}dx \, dy .
	}
	 For $x,y\in \Omega$ such that $|x-y|\geq 1$ we have that $|x-y|^{n+s{p_k}}\geq |x-y|^{n+s}$, hence
	\bgs{ 
		& \frac{1}{p_k}\frac{ |\psi_k(x)-\psi_k(y)|^{p_k} }{|x-y|^{n+s{p_k}}} 
			\leq \frac{
			C}{|x-y|^{n+s}} \; \in L^1\left( \{ (x,y)\in \Omega\times\Omega \; | \; |x-y|\geq 1\} \right),
	}
while for $|x-y|<1$, we have that $|x-y|^{n+sp_k-p_k}\geq |x-y|^{n-1+s}$, thus
	\bgs{
		\frac{1}{p_k}\frac{ |\psi_k(x)-\psi_k(y)|^{p_k} }{|x-y|^{n+s{p_k}}} 
			\leq &\;
			\frac{ 
			C \;|x-y|^{{p_k}}}{|x-y|^{n+s{p_k}}} 
			\\
			\leq&\;   
			\frac{ C} {|x-y|^{n-1+s}} 
			\; \in L^1\left( \{ (x,y)\in \Omega\times\Omega \; | \; |x-y|< 1\} \right).
		}	
	By using the Dominated Convergence Theorem, we get \eqref{smop1}.

	Now we show that
	\eqlab{ \label{smop12}
		\lim_{k\to \infty} \frac{1}{p_k}\int_\Omega \left(\int_{\Co \Omega} 
		\frac{| \psi_k(x)-\psi_k(y)|^{p_k}} {|x-y|^{n+s{p_k}}} dy\right) dx = \int_\Omega \left(\int_{\Co \Omega} \frac{| \psi(x)-\psi(y)|} {|x-y|^{n+s}} dy\right) dx .
	}
	For $x\in \Omega, y\in \Co \Omega$ and $|x-y|\geq 1$ we have that
	\bgs{
		\frac{1}{p_k}  \frac{| \psi_k(x)-\psi_k(y)|^{p_k}} {|x-y|^{n+s{p_k}}}  \leq \frac{
		C}{|x-y|^{n+s}} + \frac{2| \psi_k(y)|^{p_k}} {|x-y|^{n+s{p_k}}}   .
		}
		Hence, by \eqref{sof2} and noticing that
				\[ 
			\int_\Omega \int_{\Co \Omega\cap \{ |x-y|\geq 1\}} \frac{dx \, dy}{|x-y|^{n+s}} \leq \Per_s(\Omega, \Rn)<\infty,
			\]
			we conclude that
		\eqlab{ \label{sof1}
			&\frac{1}{p_k}\int_{\Co \Omega\cap \{ |x-y|\geq 1\} } \frac{| \psi(x)-\psi_k(y)|^{p_k}} {|x-y|^{n+s{p_k}}}dy \\ \leq &\; 
			\int_{\Co \Omega\cap \{ |x-y|\geq 1\} }\frac{C}{|x-y|^{n+s}} dy+ 
			2 \int_{\Co \Omega\cap \{ |x-y|\geq 1\} } \frac{| \psi_k(y)|^{p_k}} {|x-y|^{n+s{p_k}}} dy
			\\ 
			\leq &\; \int_{\Co \Omega\cap \{ |x-y|\geq 1\} } \frac{ C}{|x-y|^{n+s}} dy + 2 \sup_{k\in \N} \Tss^{p_k}(\psi_k,x)	\;\;	 \in L^1(\Omega).
		}

			On the other hand,
			since $p_k\searrow  1$, there exists some $s_o\in(0,1)$ such that $sp_k \leq s_o$. Then  for $|x-y| <1$,  we have that $|x-y|^{n+s{p_k}} \geq |x-y|^{n+s_o}$. Therefore we obtain
	\bgs{
		\frac{1}{p_k}  \frac{| \psi_k(x)-\psi_k(y)|^{p_k}} {|x-y|^{n+s{p_k}}}
		  \leq \frac{
		 C}{|x-y|^{n+s_o} } +  \frac{2 |\psi_k(y)|^{p_k}} {|x-y|^{n+s{p_k}}} .
	}
Hence, by \eqref{sof2} and noticing that
		\[ 
			\int_\Omega \int_{\Co \Omega\cap \{|x-y|<1\}} \frac{dx \, dy}{|x-y|^{n+s_o}} \leq \Per_{s_o}(\Omega, \Rn)<\infty,
			\]
			we get that
	\eqlab{ \label{yupy2}
	&	\frac{1}{p_k}\int_{\Co \Omega\cap \{ |x-y|\leq 1\}} \frac{| \psi_k(x)-\psi_k(y)|^{p_k}} {|x-y|^{n+s{p_k}}}  dy \\
	\leq &\;  \int_{\Co \Omega\cap \{ |x-y|< 1\}}\frac{
		 \overline c}{|x-y|^{n+s_o} } dy +  \int_{\Co \Omega\cap \{ |x-y|< 1\}}\frac{|\psi_k(y)|^{p_k}} {|x-y|^{n+s{p_k}}} dy 
		 \\
		 \leq &\; \int_{\Co \Omega\cap \{ |x-y|< 1\}}\frac{
		C}{|x-y|^{n+s_o} } dy + 2 \sup_{k\in \N} \Tss^{p_k}(\psi_k,x) \; \;  \in L^1(\Omega).
	}
			Putting together \eqref{sof1} and \eqref{yupy2}, we obtain the claim in \eqref{smop12} by employing the Dominated Convergence Theorem. 
 \end{proof}
 
\subsection{Proof of Theorem~\ref{theorem}}

With the previous preliminary work, we are now in the position of completing the proof
of Theorem~\ref{theorem}.

\begin{proof}[Proof of Theorem~\ref{theorem}]
Notice that 
	\[ 
	 \sup_{k\in\N} \Tss^{p_k}(\varphi_k,\Co \Omega; \cdot) \in {L^1(\Omega)}  \qquad \implies \qquad  \sup_{k\in\N} \|\Tss^{p_k}(\varphi_k,\Co \Omega; \cdot)\|_{L^1(\Omega)} <\overline C.
	\]
Furthermore, recalling the minimality of
$u_{p_k}$, we have that \eqref{fff34} follows by \eqref{enbdd0}.
The compactness result in  the subsequent Proposition~\ref{eqv} implies the existence of a limit function $u_1 \in \W^{s,1}_\varphi(\Omega)$ for a subsequence of $\{u_{p_k}\}_k$ (that we relabel for simplicity). We notice that, by Fatou's Lemma, 
		\eqlab{ \label{gigio1}
		 \sup_{k\in \N} \Tss^{p_k}(\varphi_k,\Co \Omega;x)  \geq&\; \liminf_{k\to \infty}\Tss^{p_k}(\varphi_k,\Co \Omega; x)= 	 \liminf_{k\to \infty} \int_{\Co \Omega} \frac{|\varphi_k(y)|^{p_k}}{|x-y|^{n+sp_k}} dy
		 \\
		  \geq&\;  \int_{\Co \Omega} \frac{|\varphi(y)|}{|x-y|^{n+s}} dy =  \Tss^1(\varphi,\Co \Omega; x) ,
			 }
			 and  that $\Tss^1(\varphi,\Co \Omega; \cdot) \in L^1(\Omega)$ follows from \eqref{uui3}. \\

			  To prove that $u_1$ is a minimizer, we proceed along these lines.\\
			  	 
We consider $\psi\colon \Rn \to \R$ to be any competitor for $u_1$ such that $ \psi \in C^\infty_c(\Omega)$ and $\psi=\varphi$ in $\Co \Omega$. We claim that
	\eqlab{
 	\label{stepa} \E^1(u_1) \leq \E^1(\psi).
 	}
We consider as a competitor for $u_{p_k}$ the function $\psi_k\colon \Rn \to \R$ defined by 
	\sys[ \psi_k :=]{& \psi && \mbox{  in } \;  \Omega,
	\\
	&\varphi_k && \mbox{  in } \;\Co \Omega.}
	  Since $\psi_k$ satisfies the assumptions of Lemma~\ref{lemmasof1}, we deduce that  
	  \begin{equation}\label{LLL}
 	 	\lim_{k\to \infty} \E^{p_k} (\psi_k)= \E^1(\psi).
 	\end{equation}
Moreover, we see that we are in the hypothesis of Lemma~\ref{liminf}, and we deduce that 
			\begin{equation}\label{MMM}
			 \E^1(u_1) \leq \liminf_{k \to \infty } \E^{p_k}(u_{p_k})
			.\end{equation}
Formulas~\eqref{LLL} and~\eqref{MMM}, combined with the fact that $u_{p_k}$ are $(s,{p_k})$-minimizers, allow us to obtain
			\bgs{
			 \E^1(u_1) \leq \liminf_{k\to \infty} \E^{p_k}(u_{p_k}) \leq \liminf_{k\to \infty}\E^{p_k}(\psi_k)= \E^1(\psi),  
			}
			which concludes the proof of \eqref{stepa}.
			
	\medskip
	
   We consider now any competitor $ v\in \W^{s,1}_\varphi(\Omega)$. By Lemma~\ref{lemmasof3}, there exists a sequence $\psi_j \in C^\infty_c(\Omega)$ and $\psi_j=\varphi $ in $\Co \Omega$ such that
 	\[  \lim_{j\to \infty} \E^1(\psi_j) = \E^1(v) .\] 
 Since \eqref{stepa} holds for any $\psi_j$ here described, we finally obtain
	\bgs{
				\E^1(u_1) \leq \lim_{j\to \infty} \E^1(\psi_j) =  \E^1(v).
			}
		This gives that $u_1$ is a minimizer, and concludes the proof of the theorem.
	\end{proof}
		
We comment now on the requirement on the exterior data in Theorem \ref{theorem}.

	\begin{lemma} \label{yup5} Let $p_k \searrow 1$ and $q\in(1,1/s)$. 
If $\varphi_k\in W^{s,q}(\Co \Omega)$ and 
	\eqlab{ \label{joi} \sup_{k\in \N} \|\varphi_k\|_{W^{s,q}(\Co \Omega)} <\overline C, 
	}
	for some $\overline C>0$,
	 then there exists $\varphi \in W^{s,q}(\Co \Omega)$ such that
	\[
		\varphi_k \xrightarrow[k\to \infty]{} \varphi \mbox{ in } L^1_{\rm{loc}}(\Co \Omega) \mbox{ and a.e. in } \;\Co \Omega,
	\]
	up to subsequences. Furthermore
	\bgs{
		\sup_{k\in \N} \Tss^{p_k}(\varphi_k,\Co \Omega; \cdot)\in L^1(\Omega) \quad \mbox{ and }\qquad  \Tss^{1}(\varphi,\Co \Omega; \cdot) \in L^1(\Omega) .
	}
	\end{lemma}
	
\begin{proof}
Reasoning as in 
the proof of Lemma \ref{yup3} and using~\eqref{joi}, we get that
	\[ 
		\int_\Omega \sup_{k\in \N} \Tss^{p_k}(\varphi_k,\Co \Omega; x)  \, dx 
			\leq   C \sup_{k\in \N}  \|\varphi_k\|_{W^{s,q}(\Co \Omega)}^q 
			+ 2\Per_s(\Omega, \Rn) +  2\Per_{sq}(\Omega, \Rn)   
			\leq C,
\]
up to renaming the constants. 

Moreover, from \eqref{joi}, given any $R>0$ such that $\Omega \Subset  B_R$, the compact embedding $W^{s,q}(B_R\setminus\Omega)\Subset L^1(B_R\setminus\Omega)$ ensures that there exists $\varphi_R\in L^1(B_R\setminus  \Omega)$ such that
	\[
	  	\|\varphi_k -\varphi_R\|_{L^1(B_R \setminus \Omega)} \xrightarrow[k\to \infty]{} 0,
	\]
	up to a subsequence. By a diagonal argument we thus obtain that
	\[
		\varphi_k \xrightarrow[k\to \infty]{} \varphi \quad \mbox{in }L^1_{loc}(\Co\Omega)\mbox{ and a.e. in } \; \Co \Omega,
	\]
	up to a subsequence. 
		Then, by Fatou's Lemma we have 
		\bgs{
			 \|\varphi\|_{W^{s,q}(\Co \Omega)}
			\leq  \liminf_{k\to \infty} \|\varphi_k\|_{W^{s,q}(\Co \Omega)}\le\overline{C},
		}
	hence $\varphi \in W^{s,q}(\Co \Omega)$. Finally, the fact that $\Tss^1(\varphi,\Co \Omega; \cdot) \in L^1(\Omega)$ follows again by Fatou's Lemma, as in \eqref{gigio1}. 
\end{proof}
	
	\section{$\Gamma$-convergence of the $(s,p)$--energy to the $(s,1)$--energy}\label{GAMMASE}

This section contains the proofs of Theorems~\ref{gammy} and~\ref{fixeddata}.
	
\subsection{$\Gamma$-convergence with and without  fixed exterior conditions}		

We focus on the $\Gamma$-convergence setting given in Theorem~\ref{fixeddata}
and we will later take into account the requested modification to address
also Theorem~\ref{gammy}.

	\begin{proof}[Proof of Theorem~\ref{fixeddata}]
	According to the sequential definition of $\Gamma$-convergence (see e.g. \cite{maso,Braid}), we have to prove
	\begin{enumerate}
	\item the liminf inequality, i.e.  let $u\in L^1(\Omega)$,  and let $p_k \searrow 1$ as $k\to \infty$, then for every sequence $(u_k)_{k\in \N}\subset L^1(\Omega)$ such that
	\[ 
		u_k \xrightarrow[k\to \infty]{} u \qquad \mbox{ in  } \; L^1(\Omega),
	\]
	it holds that
	\[
		\tilde \En_{s,\varphi}^1(u) \leq \liminf_{k\to \infty} \tilde \En_{s,\varphi}^{p_k}(u_k),
	\]
	and
	\item the existence of a recovery sequence: let $u\in L^1(\Omega)$ and let $p_k \searrow 1$ as $k\to \infty$, then there exists a sequence $(u_k)_{k\in \N}\subset L^1(\Omega)$ such that
	\[ 
			u_k \xrightarrow[k\to \infty]{} u \qquad \mbox{ in  } \; L^1(\Omega),
	\]
	and 
	\[
		\tilde\En_{s,\varphi}^1(u) = \lim_{k\to \infty} \tilde\En_{s,\varphi}^{p_k} (u_k ).	
	\]
	\end{enumerate}
	
We observe that the liminf inequality~(1) is warranted by Lemma~\ref{liminf}.

	To build the recovery sequence~(2), we proceed  as follows. We remark that if $u\in L^1(\Omega)  \setminus \mathcal X_\varphi^1(\Omega)$,  there is nothing to prove (as it is enough to consider the sequence $u_{p_k}=u$ for every $k$).
	
Let then $u\in\mathcal X_\varphi^1(\Omega)$, hence $u\in \W^{s,1}_\varphi(\Omega)$, and $p_k\searrow1$. Notice that, by \eqref{pup1}
and Lemma \ref{EQYU:FOR}, we can assume that $p_k\in(1,1/s)$ for every $k$ and
\eqlab{ \label{Teferi}
	\sup_{k\in\N } \Tss^{p_k} (\varphi,\Co \Omega; \cdot)  \in L^1(\Omega).
}

According to Lemma~\ref{lemmasof3}, since $u\in \W^{s,1}_\varphi(\Omega)$, there exists $v_1 \colon \Rn \to \R$ with $v_1\in C^\infty_c(\Omega)$ and $v_1=\varphi$ in $\Co \Omega$, such that
	\[
		\|u-v_1\|_{W^{s,1}(\Omega)} < \frac{1}{4 }\quad\mbox{and}\quad |\tilde\En_{s,\varphi}^1(v_1) -\tilde\En_{s,\varphi}^1(u)|<\frac{1}{4}.
	\]  
Since $v_1\in C^{\infty}_c(\Omega)$, by \eqref{Teferi} and Lemma~\ref{domain_lem}, we have that $v_1 \in \mathcal X_\varphi^{p_k}(\Omega)$ for every $k$.
Thanks to Lemma~\ref{lemmasof1} (applied here to $\psi_k:= v_1 $ in $\Rn$), we obtain
\[ 
	\lim_{k\to \infty} \tilde \En_{s,\varphi}^{p_k}(v_1)= \tilde \En_{s,\varphi}^1(v_1).
	\]
Therefore,  there exists $\tilde k_1>1$ such that for all $k\geq\tilde k_1$ 
	\[ 	
		|\tilde\En_{s,\varphi}^{p_k}(v_1) -\tilde\En_{s,\varphi}^1(v_1) |< \frac{1}4.
	\]
		Consequently,
	\[ 	
		|\tilde\En_{s,\varphi}^{p_k}(v_1) -\tilde\En_{s,\varphi}^1(u) |\leq 
		|\tilde\En_{s,\varphi}^{p_k}(v_1) -\tilde\En_{s,\varphi}^1(v_1) |+ 	
		|\tilde\En_{s,\varphi}^{1}(v_1) -\tilde\En_{s,\varphi}^1(u) |< \frac{1}2,
		\]
		for all $k\geq\tilde k_1$.
		
	Proceeding in the same way, we continue building the sequence $(v_l)_l$ with $v_l\in C^{\infty}_c(\Omega)$, $v_l=\varphi$ in $\Co \Omega$ (and $v_l\in \mathcal X^{p_k}_\varphi(\Omega)$ for every $k$), such that
	there exists $\tilde k_l>\tilde k_{l-1} > \dots >\tilde k_1$, with
		\[
			|\tilde\En_{s,\varphi}^{p_k}(v_l)-\tilde\En_{s,\varphi}^1(u)| < \frac{1}{2^{l}} \qquad{\mbox{and}} \qquad \|v_l -u\|_{W^{s,1}(\Omega)}\leq \frac{1}{  2^{l+1} },
	\]
	 for all $k\geq\tilde k_l$.
	Now we define 
	\bgs{
		&u_k := v_1 \qquad \forall \; k< \tilde k_2,
		\\
		&u_k :=  v_2 \qquad \forall \; k\in [\tilde k_2, \tilde k_3),
		\\
		& \dots 
		\\
		& u_k :=  v_l \qquad \forall \; k\in [\tilde k_l, \tilde k_{l+1}),
		\\
		& \dots 
		}
	and we have
	that: for all $l\in\mathbb N$ there exists $\tilde k_l$ such that for all $k\geq\tilde k_l$,
	\[
		|\tilde\En_{s,\varphi}^{p_k}(u_k)-\tilde\En_{s,\varphi}^1(u)| < \frac{1}{2^l} \quad\mbox{ and }\quad \|u_k-u\|_{W^{s,1}(\Omega)} < \frac{1}{2^{l+1}} .
		\]
	Since $l$ can be taken arbitrarily large, this implies that
		\[ \lim_{k\to \infty} \tilde\En_{s,\varphi}^{p_k}(u_k) = \tilde\En_{s,\varphi}^1(u)\quad\mbox{ and }\quad u_k \xrightarrow[k\to \infty]{} u \; \mbox{ in } L^1(\Omega),
		\]
		hence we have built the required recovery sequence $u_k\in \mathcal X^{p_k}_\varphi(\Omega)$.
	\end{proof}
	
Notice that in order to obtain the $\Gamma$-convergence result in Theorem \ref{fixeddata}, we need to require and additional condition \eqref{pup1} on the exterior data. Indeed, by constructing the next counter-example, we point out that in absence of such a condition the result fails.				
\begin{remark}[Counter-example to
Theorem~\ref{fixeddata} if assumption~\eqref{pup1}
is dropped] \label{fgr}
 Let $R\ge1$ be  such that $\Omega \Subset B_R$, let 
 \begin{equation}\label{EFFE}
{\mbox{ $f\in L^1(\R^n)$ such that $f\not\in L^p(\Co B_R)$ for every $p>1$,}}\end{equation}
and define $\varphi \colon \Co \Omega \to \R$ to be
	\sys[\varphi(x)=]{ &f(x) |x|^{n+s} && {\mbox{ if }}x\in \Co B_R,
		\\
		& 0&& {\mbox{ if }}x\in B_R\setminus \Omega.
	}
	Notice that, on the one hand,
	\bgs{
		\|\Tss^1 (\varphi,\Co \Omega; x) \|_{L^1(\Omega)}=&\;  \int_\Omega\left(\int_{\Co \Omega} \frac{ |\varphi(y)|}{|x-y|^{n+s}} dy\right)dx	
		\leq  
		C  \int_\Omega\left(\int_{\Co B_R} \frac{ |\varphi(y)|}{|y|^{n+s}} dy\right)dx
		\\
		= &\;
		C  |\Omega|\int_{\Co B_R}  |f(y)| dy	<\infty.
	}
	On the other hand, for any $p>1$, we have that
	\bgs{
		\|\Tss^p (\varphi,\Co \Omega; x) \|_{L^1(\Omega)}=&\;  \int_\Omega\left(\int_{\Co \Omega} \frac{ |\varphi(y)|^p}{|x-y|^{n+sp}} dy\right)dx
		\geq  C\int_\Omega\left(\int_{\Co B_R}  |f(y)|^p |y|^{n(p-1)} dy\right)dx\\
		&
		\geq  
		C |\Omega| \int_{\Co B_R} |f(y)|^p dy = +\infty.
	}

Now let $u:\R^n\to\R$ be any measurable function such that $u\in W^{s,1}(\Omega)$ and $u=\varphi$ almost everywhere in $\Co\Omega$.
In light of Lemma~\ref{domain_lem}, we have $\E^1(u)<\infty$. On the other hand, given any sequence $p_k\searrow1$ and $u_k:\R^n\to\R$ such that $u_k=\varphi$ almost everywhere in $\Co\Omega$ and $u_k\to u$ in $L^1(\Omega)$, we have $\E^{p_k}(u_k)=+\infty$ for every $k$ big enough. 
This 
shows that
Theorem~\ref{fixeddata} does not hold true without assumption~\eqref{pup1}.
\end{remark}

We focus now on a more general setting for the $\Gamma$-convergence, also using the
notations for~$\mathcal X^{q}(\Omega)$ and~$\tilde\En_s^{q} $ that were introduced
in~\eqref{kg-98uhn-jmnMSA-1}
and~\eqref{kg-98uhn-jmnMSA-2}.
	\begin{remark}\label{aggiunta1}	We point out that it is not restrictive to consider functions that are locally summable in $\R^n$. Indeed, if $u:\R^n\to\R$ is a measurable function such that $\En_s^{q} (u)<\infty$, for some $q\in[1,1/s)$, then $u\in L^1_{\loc}(\Rn)$. To see this, recall from Lemma~\ref{domain_lem}, that $u\in L^q(\Omega)$ and furthermore, for any $R>0$ such that $\Omega \Subset B_R$,
		\bgs{
		\int_{B_R\setminus \Omega} |u(y)|^q \, dy 
		\leq&\; C(R,s,q,\Omega)\int_{\Omega} \left(\int_{B_R\setminus \Omega}   \frac{ |u(y)|^q}{|x-y|^{n+sq}} dy \right) dx
					\\
					 \leq &\; C  \|\Tss^q(u,\Co \Omega; \cdot)\|_{L^1(\Omega)}<\infty,
		}
	so that, actually, $u\in L^q_{\loc}(\Rn)$.
	\end{remark}
\medskip

We now complete the proof of Theorem~\ref{gammy}.

	\begin{proof}[Proof of Theorem~\ref{gammy}]	According to the sequential definition of $\Gamma$-convergence (see e.g. \cite{maso,Braid}), we have to prove
	\begin{enumerate}
	\item the liminf inequality: let $u\in L^1_{\loc}(\Rn)$  and let $p_k \searrow 1$ as $k\to \infty$, then for every sequence $(u_k)_{k\in \N}\subset L^1_{\loc}(\Rn)$ such that
	\[ 
		u_k \xrightarrow[k\to \infty]{} u \qquad \mbox{ in  } \; L^1_{\loc}(\Rn),
	\]
	it holds that
	\[
		\tilde \En_s^1(u) \leq \liminf_{k\to \infty} \tilde \En_s^{p_k}(u_k),
	\]
	and
	\item the existence of a recovery sequence: let $u\in L^1_{\loc}(\Rn)$ and $p_k \searrow 1$ as $k\to \infty$, then there exists a sequence $(u_k)_{k\in \N}\subset L^1_{\loc}(\Rn)$ such that
	\[ 
			u_k \xrightarrow[k\to \infty]{} u \qquad \mbox{ in  } \; L^1_{\loc}(\Rn),
	\]
	and 
	\[
		\tilde\En_s^1(u) = \lim_{k\to \infty} \tilde\En_s^{p_k} (u_k ).	
	\]
	\end{enumerate}
	We observe that the liminf inequality follows from Lemma~\ref{liminf}, so we focus on building the recovery sequence. 
	
For this, one of the ingredients in the proof consists in the approximation
of a given function of finite energy
by smooth and compactly supported functions (this technique\footnote{We
stress that we could not employ this method directly in Theorem~\ref{fixeddata},
since in that case we need to keep the exterior data fixed and approximate
only from the inside the domain.}
was used
e.g. also in~\cite{BN16b}). To this end,
	  we remark at first that if $u\in L^1_{\loc}(\Rn)  \setminus \mathcal X^1(\Omega)$,  there is nothing to prove (as we can consider $u_{p_k}:= u$). Let 
 then $u\in\mathcal X^1(\Omega)$, hence $u\in \W^{s,1}(\Omega)$ and $\Tss^1(u,\Co \Omega; \cdot)\in L^1(\Omega)$ by Lemma~\ref{domain_lem}.
	
	Notice that we can assume that $p_k\in(1,1/2s)$ for every $k$.
	
	Now we cut $u$ at heights $-M,M$, for a fixed $M>0$. Precisely,  we define
	\[ u^M(x) = \min\left\{ M, \max\left\{-M, u(x)\right\}\right\},
	\]
	and notice that
	\eqlab{ \label{roba9}|u^M(x)| \leq |u(x)|\qquad{\mbox{and}}\qquad  |u^M(x)-u^M(y)| \leq |u(x)-u(y)|.}
	Since $u^M \to u$ as $M\to \infty$  almost everywhere in $\Rn$, and since, thanks to \eqref{roba9}, 
	\[
	\frac{|u^M(x)-u^M(y)|}{|x-y|^{n+s}} \leq \frac{|u(x)-u(y)|}{|x-y|^{n+s}} \in L^1(Q(\Omega)),\]
	from the Dominated Convergence Theorem we have that
	\eqlab{\label{time_travel0}
	\lim_{M\to \infty}\|u^M-u\|_{W^{s,1}(\Omega)}=0\quad\mbox{and}\quad\lim_{M\to \infty} \tilde \En_s^1(u^M)= \tilde \En_s^1(u).
}
	Notice also that, since $|u^{M}| \leq |u|$,  by the Dominated Convergence Theorem we obtain
	\eqlab{\label{time_travel1} \lim_{M \to \infty} \|u^M-u\|_{L^1(B_R)}=0,}
	for any $R>0$.
	Let us fix a sequence $R_k\nearrow\infty$ such that $\Omega\Subset B_{R_1}$.

	By \eqref{time_travel0} and \eqref{time_travel1}, there exists $\tilde M_1>0$ such that for all $M\geq \tilde M_1$,
	\[ |\tilde \En_s^1(u^M)- \tilde \En_s^1(u)|< \frac16\quad\mbox{and}\quad \|u^M -u\|_{L^1(B_{R_1} )} < \frac{1}{4}.
	\]
Since $u^{\tilde M_1}\in \W^{s,1}(\Omega)$, according to Lemma~\ref{lemmasof3}, there exists $v_1^{\tilde M_1} \colon \Rn \to \R$ with $v_1^{\tilde M_1}\in C^\infty_c(\Omega)$ and $v_1^{\tilde M_1}=u^{\tilde M_1}$ in $\Co \Omega$, such that
	\[
		\|u^{\tilde M_1}-v_1^{\tilde M_1}\|_{W^{s,1}(\Omega)} < \frac{1}{4}\quad\mbox{and}\quad|\tilde \En_s^1(v_1^{\tilde M_1})- \tilde \En_s^1(u^{\tilde M_1})|  <\frac16.
	\] 
Thus
	\eqlab{ \label{Narset} \|v_1^{\tilde M_1} -u\|_{L^1(B_{R_1} )}\leq\|v_1^{\tilde M_1} -u^{\tilde M_1}\|_{L^1(\Omega )}+\|u^{\tilde M_1} -u\|_{L^1(B_{R_1} )} < \frac{1}{2},}
and 
	\eqlab{ \label{roba8}
		|\tilde \En_s^1(v_1^{\tilde M_1})- \tilde\En_s^1(u)| \leq
			|\tilde \En_s^1(v_1^{\tilde M_1})- \tilde \En_s^1(u^{\tilde M_1})| + |\tilde \En_s^1(u^{\tilde M_1})- \tilde \En_s^1(u)| < \frac13.
	}
	We now observe that $v_1^{\tilde M_1}\in \mathcal X^{p_k}(\Omega)$ for every $k$. This is a consequence of the fact that $v_1^{\tilde M_1}\in C^\infty_c(\Omega)$, hence the interaction of $\Omega$ with itself provides a bounded contribution
	to the energy functional, and of the boundedness of $v_1^{\tilde M_1}$ outside of $\Omega$, which, thanks to the assumption $p_k\in(1,1/2s)$, actually ensures that
	\[ \sup_{k\in \N} \Tss^{p_k}(v_1^{\tilde M_1},\Omega) \leq \sup_{k\in \N} \int_{ \Co \Omega} \frac{{\tilde M_1}}{|x-y|^{n+sp_k}} dy \in L^1(\Omega),
	\]
	see the computations in \eqref{bleah3}.
	 Hence \eqref{sof2} is satisfied and  Lemma~\ref{lemmasof1} implies that
\[ 
	\lim_{k\to \infty} \tilde \En_s^{p_k}(v_1^{\tilde M_1})= \tilde \En_s^1(v_1^{\tilde M_1}).
	\]
	From this we deduce the existence of $\tilde k_1>0$ such that for all $k\geq \tilde k_1$,
	\[ 
		|\tilde \En_s^{p_k}(v_1^{\tilde M_1})- \tilde \En_s^1(v_1^{\tilde M_1})|<\frac16.
	\]
	Taking into account this, \eqref{Narset} and \eqref{roba8}, we have that there exist $\tilde M_1>0$ and $\tilde k_1>1$ such that
	\[
	|\tilde \En_s^{p_k}(v_1^{\tilde M_1})- \tilde \En_s^1(u)|<\frac12, \qquad \mbox{ and }  \quad \|v_1^{\tilde M_1} -u\|_{L^1(B_{R_1} )} < \frac{1}{2},
	\]
	 for all $k \geq \tilde k_1$.
	Proceeding in the same way, we continue building the sequence $\tilde M_l>\tilde M_{l-1}>\dots \tilde M_1$ and $\tilde k_l>\tilde k_{l-1} >\dots \tilde k_1$, and $v_l^{\tilde M_l}$ such that
	\[
	|\tilde \En_s^{p_k}(v_l^{\tilde M_l})- \tilde \En_s^1(u)|<\frac1{2^l}, \qquad \|v_l^{\tilde M_l} - u\|_{L^1(B_{R_l})}<\frac{1}{	 2^{l} }, \]
	for all $k\geq \tilde k_l$.
	
			We now define 
	\bgs{
		&u_k := v_1^{\tilde M_1} \qquad \forall \; k< \tilde k_2,
		\\
		& u_k :=  v_2^{\tilde M_2} \qquad \forall \; k\in [\tilde k_2, \tilde k_{3}),
		\\
		& \dots 
		\\
		& u_k :=  v_l^{\tilde M_l} \qquad \forall \; k\in [\tilde k_l, \tilde k_{l+1}),
		\\
		& \dots
	}
	and we have
	that for all $l>0$ there exists $\tilde k_l$ such that for all $k>\tilde k_l$,
	\[
		|\tilde \En_s^{p_k}(u_k)-\tilde \En_s^1(u)| < \frac{1}{2^l}  , \qquad \|u_k-u\|_{L^1(B_{R_l})} < \frac{1}{2^{l} }.
		\]
	Since $l$ can be taken arbitrarily large, and $R_l\nearrow\infty$, this implies that
		\[ \lim_{k\to \infty} \tilde \En_s^{p_k}(u_k) = \tilde \En_s^1(u), \qquad u_k \xrightarrow[k\to \infty]{} u \; \mbox{ in } L^1_{\loc}(\Rn).
		\]
		
		Thus, $(u_k)_k $ is a recovery sequence, and we conclude the proof of the desired result.
\end{proof}

	\subsection{Equi-coercivity}
	We focus now	on the equi-coercivity of the family of functionals $\tilde \E^p$ (for $p$ close enough to $1$).
We point out that this type of results and arguments have been discussed
previously in \cite[Theorem 2 and Section 3]{Ngu11} and \cite[Theorems 2 and 3]{BN18}.
	In what follows, we suppose that the exterior condition is as general as possible,
i.e., that the tails of the sequence of exterior data $\varphi_k$ are uniformly bounded.
Such a condition is satisfied, for instance, if $\varphi_k \in W^{s,q}(\Omega)$ for
some $q>1$, as proved in Lemma~\ref{yup5}. More precisely:

	\begin{proposition}[Equi-coercivity with varying exterior data]\label{eqv}
	Let $p_k \searrow 1$ as $k\to \infty$ 
	and $\varphi_k \colon \Co \Omega \to \R$ such that 
	\eqlab{ \label{gigi41}	
	 \limsup_{k\to \infty} \|\Tss^{p_k}
(\varphi_k,\Co \Omega; \cdot)\|_{L^1(\Omega)}<\infty,
\qquad \mbox{ and }\qquad  \varphi_k \xrightarrow[k\to \infty]{ } \varphi \; \mbox{ a.e. in } \Co \Omega .
 	}
 Let $u_k\in\W^{s,p_k}_{\varphi_k}(\Omega)$ such that
 \[
 \limsup_{k\to\infty} \E^{p_k}(u_k) <\infty.
 \]
 	Then there exists $u_1\in \W_\varphi^{s,1}(\Omega)$   such that, up to subsequences,
	\[ u_k \xrightarrow[k\to \infty]{} u_1 \mbox{ in } L^1(\Omega) \quad \mbox{and a.e. in } \Rn
	.\]
				\end{proposition}
	\begin{proof}
	In the rest of the proof, the constant~$C$ may change value from line to line, denoting nonetheless a positive quantity independent of $k$. 
	
	We can suppose without loss of generality that
	\eqlab{ \label{fff34}
	\sup_{k\in\N} \E^{p_k}(u_k)\le C\quad\mbox{and}\quad \sup_{k\in\N} \|\Tss^{p_k} (\varphi_k,\Co \Omega; \cdot)\|_{L^1(\Omega)} \leq C.
 	}
	Notice that the uniform bound on the energies implies that
	\[ \; \sup_{k\in\N}[u_k]_{W^{s,p_k}(\Omega)}\le C.\]
	Moreover, arguing as in \eqref{poin}, we have
	\[ \sup_{k\in \N} \|u_k\|_{L^{p_k}(\Omega)}\le C(n,\Omega)\left(1+\sup_{k\in \N} \|\Tss^{p_k}(u_k,\Co\Omega;\cdot)\|_{L^1(\Omega)}\right) \le C.\]
	Fixed $\sigma\in(0,s)$, thanks to Theorem \ref{quattro} we obtain 
	\[
		\sup_{k\in \N} \|u_k\|_{W^{\sigma,1}(\Omega)} \leq \mathfrak C_1\sup_{k\in \N} \|u_k\|_{W^{s,p_k}(\Omega)} \leq C.
	\]
	By the compact embedding $W^{\sigma,1}(\Omega)\Subset L^1(\Omega)$, there exists $u \in L^1(\Omega)$ such that
	\[
		u_k \xrightarrow[k\to \infty]{} u \quad \mbox{  in }\quad L^1(\Omega)\quad\mbox{and a.e. in }\Omega,
	\]
	up to subsequences.
	Moreover, using Fatou's lemma, we have that
	$u \in W^{s,1}(\Omega)$. 
	Then, 
	\sys[u_1:=]{ &u &&\mbox{ in } \; \Omega\\
				&\varphi &&\mbox{  in } \; \Co \Omega,}
		 provides the desired limit function. 
\end{proof}

As a direct consequence of Proposition \ref{eqv}, we have the equi-coercivity for fixed exterior data. 
				
\begin{corollary}[Equi-coercivity with fixed exterior data]\label{eqv1}
	Let $\varphi\colon \Co \Omega \to \R$ be such that 
	\[
\limsup_{p\searrow1} \|\Tss^{p} (\varphi,
\Co \Omega; \cdot) \|_{L^1(\Omega)}<\infty.
 	\]
 	Let $p_k \searrow 1$ as $k\to \infty$ 
 	and let $u_k\in\W^{s,p_k}(\Omega)$ such that
 	\[
 	\limsup_{k\to\infty} \E^{p_k}(u_k) <\infty.
 	\]
	Then there exists $u_1\in \W_\varphi^{s,1}(\Omega)$ such that, up to subsequences,
	\[ u_k \xrightarrow[k\to \infty]{} u_1 \mbox{ in } L^1(\Omega) \quad \mbox{ a.e. in } \Rn.
\]
\end{corollary}	
				
We conclude this section by pointing out some observations related to Proposition \ref{eqv}. 
\begin{remark} \label{mk2} We note that the hypothesis $\varphi_k \to \varphi$ almost 
everywhere in $\Co \Omega$ as $k\to\infty$, as stated in~\eqref{gigi41}, cannot be removed. 
Indeed, let $\varphi_k \colon \R \to [-1,1]$,		
\[ \varphi_k(x):= \sin(kx),\]
and let $\Omega\subset\R$ be any bounded open interval.

Then $\varphi_k$ satisfies, on the one hand, that
			\[ \sup_{k\in \N} \Tss^{p_k} (\varphi_k,\Co \Omega; \cdot)\in L^1(\Omega)
			,\]
			reasoning as in \eqref{bleah3}. On the other hand, no subsequence of $\varphi_k$ has a limit in $\Co\Omega$. Indeed, if there were some $\varphi$ such that $\varphi_{k_l} \to \varphi$ almost everywhere in $\Co\Omega$, with $k_l\nearrow\infty$, then
			\[ \left(\sin\left( k_{l+1} x \right) - \sin \left( k_l x\right)\right)^2 \xrightarrow[l \to \infty]{} 0\quad\mbox{a.e. in }\Co\Omega,
			\]
			hence, by the Dominated Convergence Theorem,
			\begin{equation}\label{--8tfd} \lim_{l \to \infty} \int_{2M\pi}^{(2M+1)\pi} \left(\sin\left( k_{l+1} x \right) - \sin \left( k_l x\right)\right)^2 \,dx=0,
			\end{equation}
			with $M\in\N$ big enough such that $\Omega\subset(-\infty,2M\pi)$.
			Nonetheless,
			\[\int_{2M\pi}^{(2M+1)\pi} \left(\sin\left( k_{l+1} x \right) - \sin \left( k_l x\right)\right)^2\,dx =\pi,\]
			and this provides a contradiction with~\eqref{--8tfd}. This shows that assumption~\eqref{gigi41}
			cannot be dropped.
			\end{remark}
			
\begin{remark}
Condition~\eqref{gigi41} can be slightly weakened, without
imposing a priori the existence of a limit function of the exterior data.
This can be done in two ways:  
			\begin{itemize}
			\item either one fixes the exterior data, as in Corollary \ref{eqv1},  
			\item or one requires a more restrictive condition on the exterior data, e.g. that $\varphi_k\in W^{s,q}(\Co\Omega)$, for some $q\in(1,1/s)$, with
			\[
			\sup_{k\in \N} \|\varphi_k\|_{W^{s,q}(\Co \Omega) }< \infty,
			\]
			as in Lemma~\ref{yup5}.
			\end{itemize}
\end{remark}			
				
\section{Convergence of weak solutions of the $(s,p)$--Laplacian to weak solutions of the $(s,1)$--Laplacian}  \label{tygh7t798fdwevirbg}
 
 We prove here Theorem~\ref{answer}.

For this, we focus on the energy functional
\[\E^1(u)
	=\frac{1}{2}\iint_{Q(\Omega)} \frac{|u(x)-u(y)|}{|x-y|^{n+s}} dx\, dy 
	,\]
	and we exploit the setting in Definition~\ref{s1lapdel}.

It is worth stressing that Definition \ref{s1lapdel} makes sense with no a priori assumption on $u$---besides measurability, since, by the fractional Hardy-type inequality  \eqref{fracH} and the global boundedness of $\z$, we have
\eqlab{\label{Ugin}
\left|\iint_{Q(\Omega)} \frac{\z(x,y)}{ \vert x - y \vert^{n+s}} (w(x) - w(y)) dx\, dy\right|\leq C(n,s,\Omega)\|w\|_{W^{s,1}(\Omega)}.
}

To further dwell on Definition \ref{s1lapdel}, we observe that in \eqref{ee1} we could consider just test functions $w\in C^\infty_c(\Omega)$---where it is understood that $w$ is extended by zero outside $\Omega$.\\
Indeed, if $w\in\W^{s,1}_0(\Omega)$, then we can find $w_k\in C^\infty_c(\Omega)$ such that $w_k\to w$ in $W^{s,1}(\Omega)$
(see \cite{Tri83}, Part~i) of the theorem on
page 210; see also \cite[Proposition A.1]{teolu}
 for an elementary and self-contained proof). Then, in virtue of \eqref{Ugin}, we have
	\[
	\iint_{Q(\Omega)} \frac{\z(x,y)}{ \vert x - y \vert^{n+s}} (w(x) - w(y)) dx\, dy=\lim_{k\to\infty}\iint_{Q(\Omega)} \frac{\z(x,y)}{ \vert x - y \vert^{n+s}} (w_k(x) - w_k(y)) dx\, dy=0.
	\]

Now, we prove Theorem \ref{answer}.

\begin{proof}[Proof of Theorem \ref{answer} (i)] Suppose that  $u$ is a weak solution  to the problem
\eqref{armoni}, and let  $v \in \mathcal{W}^{s,1}(\Omega)$ be such that $v = u$ almost everywhere in $\Co\Omega$. Consider the function~$w := u - v\in\W^{s,1}_0(\Omega)$, and notice that, by definition of $\z$, for almost every $(x,y)\in Q(\Omega)$ we have
\[
\z(x,y)\big(w(x)-w(y)\big)=|u(x)-u(y)|-\z(x,y)\big(v(x)-v(y)\big)\geq
|u(x)-u(y)|-|v(x)-v(y)|.
\]
By \eqref{ee1} we thereby obtain
	\bgs{
		\iint_{Q(\Omega)}\left[\frac{|u(x)-u(y)|}{|x-y|^{n+s}}-\frac{|v(x)-v(y)|}{|x-y|^{n+s}}\right]dx\,dy
		\leq
		\iint_{Q(\Omega)}  \frac{\z(x,y) (w(x)-w(y))}{ | x - y |^{n+s}}dx\, dy
		=0.
 	 	}
Given the arbitrariness of the competitor $v\in\W^{s,1}_u(\Omega)$, this implies that $u$ is an $s$-minimal function in $\Omega$.
\end{proof}
\medskip

\begin{proof}[Proof of Theorem \ref{answer} (ii)]
 Suppose now that $u$ is $s$-minimal in $\Omega$, and that $\overline{u} \in \W^{s,1}_u(\Omega)$ is a weak solution of \eqref{armoni}.
That is, recalling Definition~\ref{s1lapdel},
there exists $\z \in L^\infty(Q(\Omega))$, with $\Vert \z \Vert_\infty \leq 1$, $\z$ antisymmetric satisfying
	\begin{equation}\label{18BIS}
		\iint_{Q(\Omega) } \frac{\z(x,y)(w(x) - w(y)) }{ | x - y |^{n+sp}} dx\, dy =0 \quad \hbox{for all} \ w \in \W^{s,1}_0(\Omega),
	\end{equation}
and
	\[
	\z(x,y) \in \sgn(\overline{u}(x) - \overline{u}(y)) \quad \hbox{for almost all} \ (x, y) \in Q(\Omega).
	\]
Considering  $w:=\overline{u}-u \in \W_0^{s,1}(\Omega)$ as a test function in \eqref{18BIS}, and exploiting the $s$-minimality of $u$, we get  that
	 \bgs{
	 0=\;\iint_{Q(\Omega)}  \frac{\z(x,y) (w(x)-w(y))}{ | x - y |^{n+s}}dx\, dy
	 &=\;\iint_{Q(\Omega)}\left[\frac{|\overline{u}(x)-\overline{u}(y)|}{|x-y|^{n+s}}-\frac{\z(x,y)(u(x)-u(y))}{|x-y|^{n+s}}\right]dx\,dy\\
	 &\geq\;
	 \iint_{Q(\Omega)}\left[\frac{|\overline{u}(x)-\overline{u}(y)|}{|x-y|^{n+s}}-\frac{|u(x)-u(y)|}{|x-y|^{n+s}}\right]dx\,dy
	 \geq\;0.
	  }
It follows that
\[ \z(x,y) \in \sgn(u(x) - u(y)) \quad \hbox{for almost all} \ (x, y) \in Q(\Omega).\]
Consequently, $u$ is a weak solution  to the problem
\eqref{armoni}.
\end{proof}

It remains now to prove that a weak solution of \eqref{ee1} exists, that is to prove Theorem \ref{answer} (iii). To do this, roughly speaking, one considers a weak $(s,p)$-solution and then carefully studies the limit as~$p \searrow 1$.
Exploiting such a technique, the weak formulation for the $(s,1)$-Laplacian was introduced in \cite{toled}. In \cite[Theorem 3.4]{toled}, the existence of a solution to problem \eqref{armoni} is indeed obtained as the limit case of the solution of the Dirichlet problem for the fractional $p$-Laplacian, as $p \searrow 1$. 

To implement this strategy, we sum up the existence (and uniqueness) of minimizers, and the equivalence between minimizers
and weak solutions for $p>1$ in the following statement.  
 
 \begin{proposition}\label{deellos}
 	Let $p\in[1,1/s)$ and let $\varphi \colon \Co \Omega \to \R$ be such that
	$\Tss^p(\varphi,\Co \Omega; \cdot) \in L^1(\Omega)$.
Then, there exists an $(s,p)$-minimizer $u\in \W_\varphi^{s,p}(\Omega)$.
Moreover, if $p\in(1,1/s)$, then
\begin{itemize}
	\item[(i)] the $(s,p)$-minimizer is unique.
	\item[(ii)] A function $u \in \W^{s,p}_\varphi(\Omega)$  is an $(s,p)$--minimizer if and only if it is a weak solution of \eqref{Dirichlet1}.
\end{itemize}
\end{proposition}

A few comments about Proposition~\ref{deellos} are in order.
The proof of the existence of a minimizer can be carried out by direct methods, exploiting the estimate \eqref{poin}---see for instance \cite[ Remark~4, Theorem~5]{bucy},  \cite[Theorem~A.1]{bdlv20} for the case $p=1$, and also \cite{dkplocal} for different conditions on the exterior data. The uniqueness for $p> 1$ is ensured by the strict convexity of the operator. On the other hand, as proved in \cite[Theorem 1.6]{bdlv20}, when $p=1$ the uniqueness, in general, fails.

The equivalence of minimizers and weak solutions follows by the same argument  of \cite[Theorem~2.3.]{dkplocal} (where the formulation appears to be slightly different than ours). 

For $p\geq 1/s$, the condition $\Tss^p(\varphi,\Co \Omega; \cdot) \in L^1(\Omega)$ is still enough to ensure the existence of an $(s,p)$-minimizer---and its uniqueness is again ensured by the strict convexity of the functional. However, while for $p<1/s$ this condition is very mild and imposes no requirement on the behavior of $\varphi$ across the boundary of $\Omega$, when $sp\geq1$, roughly speaking, it forces $\varphi$ to be close to zero near $\partial\Omega$.
Moreover, when $p\geq 1/s$, Lemma~\ref{domain_lem} no longer holds true. In particular, even if $\Tss^p(\varphi,\Co \Omega; \cdot) \in L^1(\Omega)$, a function $u\in\W^{s,p}_\varphi(\Omega)$ might be such that $\En_s^p(u)=+\infty$. As a consequence, point (ii) does not hold in these hypothesis. 

We also point out that the case~$sp\ge1$ is interesting
due to its connection to the classical case in the limit as~$s\nearrow1$.

\bigskip 

We prove now the existence of a weak solution of \eqref{ee1}. 
We observe that in our formulation we take  slightly different
hypotheses than those in \cite[Theorem 3.4]{toled}.
Precisely,  the authors of  \cite{toled} assume zero exterior condition (corresponding to $\varphi=0$) (but consider an $L^2$ right hand side in \eqref{ee1}). This difference, in our case, translates into some extra care when dealing with interactions $\Omega$ with $\Co \Omega$ once the suitable conditions on $\varphi$ are set, hence in some additional computations (that we carry out in the last part of our proof).  Besides these differences, the proof is
close to the one of   \cite[Theorem 3.4]{toled}, and we insert it for completeness.

\begin{proof}[Proof of Theorem \ref{answer} (iii)]
We consider a sequence $p_k \searrow 1$ as $k\to\infty$. We can assume without loss of generality that $p_k\in (1, q]$, which implies that $s_{p_k}p_k<1$ for every $k$.

We begin by observing that, arguing as in Lemma \ref{EQYU:FOR}, 
\eqref{new} is equivalent to 
\eqlab{ \label{foglia} 	\sup_{p\in (1,q]} \Ts_{s_p}^{p}(\varphi,\Co \Omega; \cdot) \in L^1(\Omega) }
and to
\eqlab{ \label{albero} \sup_{p\in (1,q]} \| \Ts_{s_p}^{p}(\varphi,\Co \Omega; 
\cdot) \|_{L^1(\Omega) }<\infty.}
Hence, by \eqref{albero}, we can apply Proposition \ref{deellos}, which ensures that there exists a sequence of unique $(s_{p_k},p_k)$-minimizers $u_{p_k} \in \W^{s_{p_k},p_k}_\varphi(\Omega)$, that are also weak solutions of  

\syslab[]{ \label{Dirichlet2}
 & (-\Delta)_{p_k}^{s_{p_k}} u_{p_k} =0 & &\mbox{ in } \; \Omega,
 \\ 
 & u_{p_k} = \varphi & &\hbox{ on }\;  \Co \Omega.
 }
As a consequence of \eqref{never_enough_remarks} and \eqref{enbdd0}, by \eqref{albero} we also have that
\eqlab{ \label{COTA1}
	\sup_{k\in \N}\|u_{p_k}\|_{W^{s_{p_k},p_k}(\Omega)}<\overline C
	\quad\mbox{and}\quad
	\sup_{k\in \N} 2p_k\En_{s_{p_k}}^{p_k}(u_{p_k}) <\overline C. 
}

Now we fix an index $\sigma\in(0,s)$ and we consider the constant $\mathfrak C_1(n,s,\sigma,\Omega)$ of Theorem \ref{quattro}. Since $s_{p_k}>s$ for every $k$ and $s_{p_k}\searrow s$, a careful inspection of the proof of Theorem \ref{quattro} reveals that
\[
\mathfrak C_2(n,s,\sigma,\Omega):=\sup_{k\in\N}\mathfrak C_1(n,s_{p_k},\sigma,\Omega)<\infty.
\]
By \eqref{COTA1} and Theorem \ref{quattro}, this implies that
\[
\sup_{k\in \N}\|u_{p_k}\|_{W^{\sigma,1}(\Omega)}<\infty.
\]
Hence, by compactness, there exists a subsequence of $p_k$ (that we relabel for simplicity) such that
\eqlab{ \label{ui}
	u_{p_k} \xrightarrow[k \to \infty]{} u \qquad \mbox{ in } \;  L^1(\Omega)\quad\mbox{and a.e. in }\Omega.
}
We extend $u$ to the whole of $\R^n$ by setting $u|_{\Co\Omega}:=\varphi$.

Also, by \eqref{COTA1} and Fatou's Lemma, we get that~$u \in \W^{s,1}_\varphi(\Omega)$. We now proceed to prove that $u$ is an $(s,1)$-minimizer, arguing as in the proof of Theorem \ref{theorem}.

Given $v\in\W^{s,1}_\varphi(\Omega)$, by Lemma~\ref{lemmasof3} there exists a sequence of functions $\psi_j:\R^n\to\R$ such that $\psi_j|_{\Omega}\in C^\infty_c(\Omega)$ and $\psi_j=\varphi$ almost everywhere in $\Co\Omega$, with
\[
\lim_{j\to\infty}\|\psi_j-v\|_{W^{s,1}(\Omega)}=0\quad\mbox{and}\quad
\lim_{j\to\infty}\En_s^1(\psi_j)=\En_s^1(v).
\]
Since $\psi_j|_{\Omega} \in C^\infty_c(\Omega)$, under the hypothesis \eqref{foglia}, by a minor modification of the proof of Lemma~\ref{lemmasof1}, we have that
\[
\lim_{k\to\infty}\En_{s_{p_k}}^{p_k}(\psi_j)=\En_s^1(\psi_j).
\] 
Since $\psi_j\in\W^{s_{p_k},p_k}_\varphi(\Omega)$ for every $j$, by $(s_{p_k},p_k)$-minimality of $u_{p_k}$ and exploiting Fatou's Lemma, we have that
\[
\En_s^1(u)\leq\liminf_{k\to\infty}\En_{s_{p_k}}^{p_k}(u_k)
\leq\lim_{k\to\infty}\En_{s_{p_k}}^{p_k}(\psi_j)=\En_s^1(\psi_j),
\]
for every $j$. Hence, passing to the limit as~$j\to\infty$,
\[
\En_s^1(u)\leq\liminf_{j\to\infty}\En_s^1(\psi_j)=\En_s^1(v).
\]
Given the arbitrariness of $v\in\W^{s,1}_\varphi(\Omega)$, this concludes the proof of the $(s,1)$-minimality of $u$.

\bigskip

We consider $M>0$ and we define
$$C_{p_k,M}:=\left\{(x,y)\in Q(\Omega) \; \Big| \; \left|\frac{  u_{p_k}(x) - u_{p_k}(y) }{\vert x -y \vert^{n+s}} \right|>M\right\}.$$
Then, recalling that $s_p= s+n -n/p$, we notice that for any $p$ one has~$n+s_p p  = (n+s)p$, hence
	\bgs{
		M^{p_k} |C_{p_k,M}| \leq &\;  \iint_{C_{p_k,M}} \left(\frac{ |u_{p_k}(x)-u_{p_k}(y)|}{|x-y|^{n+s} } \right)^{p_k} dx \, dy
		\\
		=&\; 
		\iint_{C_{p_k,M}} \frac{ |u_{p_k}(x)-u_{p_k}(y)|^{p_k}}{|x-y|^{n+s_{p_k} {p_k} } }  dx \, dy\leq 2p_k\En_{s_{p_k}}^{p_k}(u_{p_k}),
	} 
and by \eqref{COTA1}, this yields that
\begin{equation}\label{D2325}
|C_{p_k,M}|\le \frac{\overline C}{M^{p_k}}.
\end{equation}
On the other hand,
	\[\left|  \frac{ | u_{p_k}(x) - u_{p_k}(y) |^{{p_k}-2} ( u_{p_k}(x) - u_{p_k}(y))}{| x -y|^{(n+s)(p_k-1)}} 						\chi_{Q(\Omega)\setminus C_{p_k,M}} (x,y)\right|
	\leq \;M^{p_k-1}	,\]
	hence, since $p_k\searrow 1$, the left hand side is uniformly bounded independently of $k$.  Thus, for any $M \in \N$, there exists 
a subsequence of $\{p_k\}_k$, denoted $\{p_{k}^M\}_k$, such that
	\eqlab{ \label{leaf2} \frac{ | u_{p^M_k}(x) - u_{p^M_k}(y) |^{{p^M_k}-2}( u_{p^M_k}(x) - u_{p^M_k}(y)) }{| x -y|^{(n+s)(p^M_k-1)}}  						\chi_{Q(\Omega)\setminus C_{p^M_k,M}} (x,y) \rightharpoonup \z_{M}(x,y) ,}
 as $k\to \infty$, weakly$^*$ in $L^{\infty}(Q(\Omega))$,
with $\z_M$ antisymmetric such that $$\Vert \z_{M}\Vert_{L^{\infty}(Q(\Omega))}\le 1.$$
Furthermore, there exists a subsequence of $\{\z_{M}\}_M$ (which we still call $\{\z_{M}\}_M$ with a slight abuse of notation) such that,
\eqlab{ \label{leaf4}
	\z_{M} \; {\rightharpoonup} \; \z \qquad 
	\hbox{ 		weakly$^*$ in } \; L^{\infty}(Q(\Omega)), \; \mbox{ as } \, {M\to\infty},
	}
with $\z$ antisymmetric and $$\Vert \z \Vert_{L^{\infty}(Q(\Omega))}\le 1.$$

We claim that, for any $w\in C^\infty_c(\Omega)$,
\eqlab{ \label{claim1}
	\lim_{M \to \infty} \bigg[\lim_{k\to \infty} &\iint_{Q(\Omega)} \frac{\vert u_{p^M_k}(x) - u_{p^M_k}(y) \vert^{{p^M_k}-2} (u_{p^M_k}(x) - u_{p^M_k}(y))(w(x) - w(y)) }{\vert x - y \vert^{n+s_{p^M_k} {p^M_k}}}  dx\, dy \bigg]\\
	=&\;  \int_{Q(\Omega)} \frac{ \z(x,y)(w(x)-w(y))}{|x-y|^{n+s}} dx \, dy. 
}
To this end,
let us fix $M\in \mathbb{N}$, and denote by $C$ different constants, possibly depending on $n,s,q,\Omega,w$, but always independent of $k,M$. 
We notice that, by H\"{o}lder's inequality  we have that 
 \eqlab{ \label{car2}
 & \left| \iint_{Q(\Omega)} \frac{| u_{p^M_k}(x) - u_{p^M_k}(y) |^{{p^M_k}-2}(u_{p^M_k}(x) - u_{p^M_k}(y)) (w(x) - w(y))}{| x - y |^{n+s_{p^M_k} {p^M_k}}} \chi_{ C_{p^M_k,M}}  (x,y) dx\, dy \right|
 \\
  \leq &\;
 \left(\iint_{Q(\Omega)}  \frac{| u_{p^M_k}(x) - u_{p^M_k}(y) |^{{p^M_k}}}{| x - y |^{n+s_{p_k}{p^M_k}}} dx \, dy  \right)^{\frac{p^M_k-1}{p^M_k}}
 \left(\iint_{Q(\Omega)}  \frac{| w(x) -w(y) |^{p^M_k}}{| x - y |^{n+s_{p^M_k}{p^M_k}}}  \chi_{C_{p^M_k,M}} (x,y) dx \, dy  \right)^{\frac{1}{p^M_k}} 
 \\
 \leq &\;  C \left(\iint_{Q(\Omega)}  \frac{| w(x) -w(y) |^{p^M_k}}{| x - y |^{n+s_{p^M_k}{p^M_k}}}  \chi_{C_{p^M_k,M}} (x,y) dx \, dy  \right)^{\frac{1}{p^M_k}}  ,
}
by \eqref{COTA1}. 
Using again H\"{o}lder's inequality, \label{necsp}
 \bgs{
 		\iint_{Q(\Omega)}  \frac{| w(x) -w(y) |^{p^M_k}}{| x - y |^{n+s_{p^M_k} {p^M_k}}}  \chi_{C_{p^M_k,M}} (x,y) dx \, dy 
 			 \leq 					
 			\left(	\iint_{Q(\Omega)}  \frac{| w(x) -w(y) |^{q}}{| x - y |^{n+ s_{q} q} }dx \, dy \right)^{\frac{p^M_k}{q}} |C_{p^M_k,M}|^{\frac{q-p^M_k}{q}} ,
 }
where we have used that 
$ (n+s_{p^M_k} p^M_k )=(n+s) p^M_k$ and $n+  s_q q= (n+s)q.$
Using also \eqref{fracH}, we have that 
	\bgs{
	\iint_{Q(\Omega)}  \frac{| w(x) -w(y) |^{q}}{| x - y |^{n+ s_{q} q} }dx \, dy  \leq &\;  C  \|w\|^q_{W^{ s_q, q}(\OMega)}.
	}
 We point out that $ \|w\|^q_{W^{ s_q, q}(\Omega)} $ is finite, since $w\in C^\infty_c(\Omega)$.
 Substituting  into \eqref{car2}, 
 we have
 	\eqlab{ \label{Car1}
 	  &\left| \iint_{Q(\Omega)} \frac{| u_{p_k^M}(x) - u_{p_k^M}(y) |^{{p_k^M}-2}(u_{p_k^M}(x) - u_{p_k^M}(y)) (w(x) - w(y))}{| x - y |^{n+s_{p_k^M}{p_k^M}}} \chi_{ C_{p_k^M, M}}  (x,y) dx\, dy \right|
 	 \\
 	 &\;   \leq  C  |C_{p_k^M, M}|^{\frac{q-p_k^M}{q p_k^M}} \leq   C\,M^{-\frac{q-p_k^M}{q} },
 	 }
where the last inequality follows from \eqref{D2325}. It follows that
\eqlab{ \label{gas1}
	\lim_{M \to \infty}&\bigg[ \lim_{k\to \infty} \iint_{Q(\Omega)} \frac{| u_{p_k^M}(x) - u_{p_k^M}(y) |^{{p_k^M}-2}(u_{p_k^M}(x) - u_{p_k^M}(y)) (w(x) - w(y))}{| x - y |^{n+s_{p_k^M}{p_k^M}}} \chi_{ C_{p_k^M, M}}  (x,y) dx\, dy \bigg]\\
	=&\; 0.
}

 On the other hand, noticing that $|x-y|^{n+s_{p_k^M} p_k^M} =| x - y |^{(n+s)(p_k^M-1)} |x-y|^{n + s} $,  recalling \eqref{leaf2}, and pointing out that $(w(x)-w(y))/|x-y|^{n+s}\in L^1(Q(\Omega))$, by taking the limit as $k \to \infty$ 
 \bgs{
 	& \lim_{k\to \infty}  \iint_{Q(\Omega)} \frac{| u_{p_k^M}(x) - u_{p_k^M}(y) |^{{p_k^M}-2}(u_{p_k^M}(x) - u_{p_k^M}(y)) (w(x) - w(y))}{| x - y |^{n+s_{p_k^M} {p_k^M}}} \chi_{Q(\Omega) \setminus C_{p_k^M,M}}  (x,y) dx\, dy   
 	\\ 
 	=&\;  \lim_{k\to \infty}  \iint_{Q(\Omega)} \frac{| u_{p_k^M}(x) - u_{p_k^M}(y) |^{{p_k^M}-2}(u_{p_k^M}(x) - u_{p_k^M}(y)) }{| x - y |^{(n+s_{p_k^M})( {p_k^M}-1)}} \chi_{Q(\Omega) \setminus C_{p_k^M,M}}  (x,y)   \frac{w(x) - w(y)}{|x-y|^{n+s}} dx\, dy  
 	\\
 	=&\; 
 \iint_{Q(\Omega)} \frac{\z_M(x,y) (w(x) - w(y))}{| x - y |^{n+s}}   dx\, dy.}
We take now the limit as~$ M \to \infty$ counting on \eqref{leaf4} and obtain
 \bgs{
 	\lim_{M\to \infty} \bigg[\lim_{k\to \infty} & \iint_{Q(\Omega)} \frac{| u_{p_k^M}(x) - u_{p_k^M}(y) |^{{p_k^M}-2}(u_{p_k^M}(x) - u_{p_k^M}(y)) (w(x) - w(y))}{| x - y |^{n+s_{p_k^M} {p_k^M}}} \chi_{Q(\Omega) \setminus C_{p_k^M,M}}  (x,y) dx\, dy \bigg]  
 	\\ =&\;  
\iint_{Q(\Omega)} \frac{\z(x,y) (w(x) - w(y))}{| x - y |^{n+s}}   dx\, dy.}
 This, together with \eqref{gas1}, proves \eqref{claim1}.

 Since $u_{p_k}$ is a weak solution to the Dirichlet problem  \eqref{Dirichlet2},
 we have for any $w\in C^\infty_c(\Omega)$ that
\eqlab{ \label{new2}
		\iint_{Q(\Omega)} \frac{\vert u_{p_k}(x) - u_{p_k}(y) \vert^{{p_k}-2}}{\vert x - y \vert^{n+s_{p_k} {p_k}}} (u_{p_k}(x) - u_{p_k}(y))(w(x) - w(y)) dx\, dy =0,
			}
hence, by \eqref{claim1},
\eqlab{\label{leaf1} 	\iint_{Q(\Omega)} \frac{\z(x,y) (w(x) - w(y))}{| x - y |^{n+s}}   dx\, dy=0. }
To obtain the above equality for any $w\in \W^{s,1}_0(\Omega)$, it is enough to use the density of $C^\infty_c(\Omega)$ into $\W^{s,1}_0(\Omega)$.
 	 Indeed, let $w\in \W^{s,1}_0(\Omega)$, then there exists $w_j \in C^\infty_c(\Omega) $ such that $\|w-w_j\|_{W^{s,1}(\Omega)} \to 0$ as $j \to \infty$. Consequently,
 	 \bgs{
 	 	& \left| \iint_{Q(\Omega)} \frac{\z(x,y) (w(x) - w(y))}{| x - y |^{n+s}}   dx\, dy  - \iint_{Q(\Omega)} \frac{\z(x,y) (w_j(x) - w_j(y))}{| x - y |^{n+s}}   dx\, dy \right| 
 	 \\
 	 	\leq&\; 
 	 	\iint_{Q(\Omega)} \frac{ |(w-w_j) (x) - (w-w_j)(y)|}{| x - y |^{n+s}}   dx\, dy  \leq C \|w-w_j\|_{W^{s,1}(\Omega)},
 	 	 	 }
 	 	 	 using also that $w-w_j =0 $ in $\Co \Omega$, and \eqref{fracH}. 
 	 	 	Passing to the limit as $j\to \infty$, it follows that \eqref{leaf1} holds for any $w\in \W^{s,1}_0(\Omega)$.
 	 	 	
 	 	 	It remains to prove that $\z(x,y)\in \sgn(u(x)-u(y))$ to conclude the proof of the theorem, so it remains to prove that
 	 	 	\eqlab{ \label{new1}
 	 	 	\iint_{Q(\Omega)} \frac{\z(x,y) (u(x)-u(y))}{|x-y|^{n+s}} dx \, dy \geq \iint_{Q(\Omega)} \frac{|u(x)-u(y)|}{|x-y|^{n+s}} dx \, dy 
 	 	 	.
 	 	 	} 
 	 	 	To do this, we consider $\overline \varphi\colon \Rn \to \R$, with $\overline \varphi = \varphi $ in $\Co \Omega$ and  $\overline \varphi = 0$ in $\Omega$. For a fixed $M>0$ using again the
 	 	 	definition of the set~$C_{p_k^M,M}$, we consider the sequence $u_{p_k^M}$ satisfying \eqref{leaf2}. In \eqref{new2}, we take  $w:=u_{p_k^M} -\overline \varphi$, and obtain that
 	 	 	\bgs{
 	 	0=  &\;	\iint_{Q(\Omega)} \frac{|u_{p_k^M}(x) -u_{p_k^M}(y)|^{p_k^M}}{|x-y|^{n+s_{p_k^M} {p_k^M}}} dx \, dy 
 	 	  \\
 	 	  &\; +2 \int_{\Omega} \left(\int_{\Co \Omega} \frac{ |u_{p_k^M}(x) -u_{p_k^M}(y)|^{{p_k^M}-2 }(u_{p_k^M}(x) -u_{p_k^M}(y)) \varphi(y)}{|x-y|^{n+s_{p_k^M}{p_k^M}}} dy \right) dx 
 	 	 	.}
 	 	 	By Fatou's Lemma
 	\[\iint_{Q(\Omega)} \frac{|u(x) -u(y)|}{|x-y|^{n+s}} dx \, dy \leq  \liminf_{k\to \infty} \iint_{Q(\Omega)} \frac{|u_{p_k^M}(x) -u_{p_k^M}(y)|^{p_k^M}}{|x-y|^{n+s_{p_k^M} {p_k^M}}} dx \, dy .
 	\]
 	Now we prove that
 	\eqlab{
 	\label{leaf6}
 	\lim_{M\to \infty} &\left[ \lim_{k\to \infty} \int_{\Omega} \left(\int_{\Co \Omega} \frac{ |u_{p_k^M}(x) -u_{p_k^M}(y)|^{{p_k^M}-2 }(u_{p_k^M}(x) -u_{p_k^M}(y)) \varphi(y)}{|x-y|^{n+s_{p_k^M}{p_k^M}}} dy \right) dx  \right] 
 	\\
 	=&\;   \int_{\Omega}\int_{\Co \Omega} \frac{\z(x,y)\varphi(y)}{|x-y|^{n+s}} dx \, dy
 	.} 	
 	To prove \eqref{leaf6}, we observe that
 	\bgs{
 	& \int_{\Omega} \left(\int_{\Co \Omega} \frac{ |u_{p_k^M}(x) -u_{p_k^M}(y)|^{{p_k^M}-2 }(u_{p_k^M}(x) -u_{p_k^M}(y)) \varphi(y)}{|x-y|^{n+s_{p_k^M}{p_k^M}}} dy \right) dx \\
 	 	=&\;
 	\int_{\Omega} \left(\int_{\Co \Omega} \frac{ |u_{p_k^M}(x) -u_{p_k^M}(y)|^{{p_k^M}-2 }(u_{p_k^M}(x) -u_{p_k^M}(y)) \varphi(y)}{|x-y|^{n+s_{p_k^M}{p_k^M}}} \chi_{C_{p_k^M,M}} (x,y) dy \right) dx  
 	\\
 	 &\; + 	\int_{\Omega} \left(\int_{\Co \Omega} \frac{ |u_{p_k^M}(x) -u_{p_k^M}(y)|^{{p_k^M}-2 }(u_{p_k^M}(x) -u_{p_k^M}(y)) }{|x-y|^{(n+s)(p_k^M-1)} }
 	  \chi_{(\Omega \times \Co \OMega) \setminus C_{p_k^M,M}} (x,y) 
 	   \frac{\varphi(y)}{|x-y|^{n+s}} dy  \right) dx.  
 	}
 	Reasoning as in \eqref{car2} and~\eqref{Car1}, we get that
 	\bgs{
 	& \left| \int_{\Omega} \left(\int_{\Co \Omega} \frac{ |u_{p_k^M}(x) -u_{p_k^M}(y)|^{{p_k^M}-2 }(u_{p_k^M}(x) -u_{p_k^M}(y)) \varphi(y)}{|x-y|^{n+s_{p_k^M}{p_k^M}}} \chi_{C_{p_k^M,M}} (x,y) dy \right) dx \right|
 	\\ 
 	  	\leq&\;  \overline C \left[\int_\Omega \left(\int_{\Co \Omega} \frac{|\varphi(y)|^{q}}{|x-y|^{n+s_q q}} dy\right) dx  \right]^{\frac{1}{q}} |C_{p_k^M,M}|^{\frac{q-p_k^M}{q p_k^M}} 
 	  	\\
 	= &\;     C \|\Ts_{s_q}^q(\varphi,\Co \Omega; \cdot) \|_{L^1(\Omega)}^q M^{-\frac{q-p_k}{q}}
 	\\
 	\leq &\; C \, M^{-\frac{q-p_k}{q}}
 	,} 
 	making use of \eqref{new}. 
Recalling that, again from \eqref{new}, $\varphi(y)/|x-y|^{n+s}  \in L^1(\Omega\times \Co \Omega)$, it follows that
\[ \lim_{M\to \infty} \left(\lim_{k \to \infty} 	 \int_{\Omega} \int_{\Co \Omega} \frac{ |u_{p_k^M}(x) -u_{p_k^M}(y)|^{{p_k^M}-2 }(u_{p_k^M}(x) -u_{p_k^M}(y)) \varphi(y)}{|x-y|^{n+s_{p_k^M}{p_k^M}}} \chi_{C_{p_k^M},M} (x,y) dy\,  dx \right) =0.\]
 	
On the other hand, making use of \eqref{leaf2} and~\eqref{leaf4}, we have that
\bgs{
		& \lim_{M\to \infty} \left( \lim_{k\to \infty} \int_{\Omega} \int_{\Co \Omega} \frac{ |u_{p_k^M}(x) -u_{p_k^M}(y)|^{{p_k^M}-2 }(u_{p_k^M}(x) -u_{p_k^M}(y)) }{|x-y|^{(n+s)(p_k^M-1)} }
 	  \chi_{(\Omega \times \Co \OMega) \setminus C_{p_k^M,M}} (x,y) 
 	   \frac{\varphi(y)}{|x-y|^{n+s}} dy  \, dx \right) 
 	   \\&\;  = 
 	    \int_{\Omega} \left(\int_{\Co \Omega} \frac{\z(x,y)\varphi(y)}{|x-y|^{n+s}} dy\right) dx 
 	    }
 This proves \eqref{leaf6}.
 
{F}rom this, we obtain that
 	\eqlab{ \label{leaf5} 
 	0 \geq \iint_{Q(\Omega)} \frac{|u(x) -u(y)|}{|x-y|^{n+s}} dx \, dy +2 \int_{\Omega}\int_{\Co \Omega} \frac{\z(x,y)\varphi(y)}{|x-y|^{n+s}} dx \, dy.
 	}
 
 Now, in \eqref{leaf1}, we use $w:= u-\overline \varphi $ and obtain that
 	\[ \iint_{Q(\Omega)} \frac{\z(x,y)(u(x) -u(y))}{|x-y|^{n+s}} dx \, dy + 
 	2\int_\Omega \left(\int_{\Co \Omega}  \frac{\z(x,y) \varphi(y)}{|x-y|^{n+s}} dy \right) dx =0
 	.\] 
 	Together with \eqref{leaf5}, this implies \eqref{new1}, and the proof of the theorem.
\end{proof}

\appendix
\section{Asymptotics as~$s\to1^-$} \label{appendixxx}

In this Appendix we focus on the $(s,1)$-energy, and study the asymptotics as $s\nearrow 1$, proving convergence of  $(s,1)$-minimizers to functions of least gradient, both in a pointwise sense and using a $\Gamma$-convergence approach. 

We point out that the asymptotics as $s \nearrow 1$ of fractional seminorms,
nonlocal energies, the fractional perimeter and~$s$-minimal sets, both in a pointwise and in the $\Gamma$-convergence sense, have been studied in numerous papers. 
See~\cite{regularity,uniform,BBM,brect,davila,gammaconv,ponce,spectr1,spectr2,ngu,poncespector,Ngu06,NS19}
for further references and details. In particular,
we refer to Theorem~1.15 in~\cite{ponce} for related results.
We stress that the aforementioned papers deal 
with an energy defined either on $\Omega \times \Omega$ or on $\Rn \times \Rn$, 
with the exception of \cite{regularity,uniform,gammaconv}, which deal only 
with characteristic functions. Since our energy is defined on $Q(\Omega)$, 
and for functions that are not necessarily characteristic functions of some set,
some extra care must be taken in order to treat the interaction $\Omega \times \Co \Omega$.
As a matter of fact, our results, which we include for completeness, do not plainly follow from
known results, even though they are principally based on \cite{davila,gammaconv}.
\medskip

We start by proving the regularity result in Theorem~\ref{thmreg}.

\begin{proof}[Proof of Theorem~\ref{thmreg}]
	It was recently proved in \cite{bdlv20} that if $u$ is an $s$-minimal function, the level set $\{u\geq\lambda\}$ is $s$-minimal in $\Omega$ for every $\lambda\in\R$. On these grounds, we now establish \eqref{linftysupest} and \eqref{bvlocest} by making use of uniform volume and perimeter estimates that hold for $s$-minimal sets.
	
	We prove only the first inequality in \eqref{linftysupest}, since
	the second one is obtained in a similar way. For this, we observe that
	\[
	\sup_{\Omega'}u\leq\sup_{\Omega'}u_+=\sup\left\{\lambda>0\,|\,\big|\{u\geq\lambda\}\cap\Omega'\big|>0\right\}.
	\]
	Clearly, we can not have that~$\big|\{u\geq\lambda\}\cap\Omega'\big|=|\Omega'|$ for every $\lambda>0$, as otherwise $u=+\infty$ in $\Omega'$. Hence,
	\eqlab{\label{Indra}
	\sup_{\Omega'}u\leq \sup\left\{\lambda>0\,|\,0<\big|\{u\geq\lambda\}\cap\Omega'\big|<|\Omega'|\right\}.
}
	Now we observe that if $\lambda>0$ is such that
	\[
	0<\big|\{u\geq\lambda\}\cap\Omega'\big|<|\Omega'|
	\]
	then there exists $x_\lambda\in\big(\partial\{u\geq\lambda\}\big)\cap\Omega'$. Therefore, if we denote by~$d:=\dist(\Omega',\partial\Omega)$, we have that $\{u\geq\lambda\}$ is $s$-minimal in $B_d(x_\lambda)\subset\Omega$, and the uniform density estimates \cite[Theorem 4.1]{nms} give that
	\[
	\left|\{u\geq\lambda\}\cap B_d(x_\lambda)\right|\geq c(n,s)d^n.
	\]
	Thus
	\eqlab{\label{Grixis}
	\|u_+\|_{L^1(\Omega)}\geq \lambda \left|\{u\geq\lambda\}\cap B_d(x_\lambda)\right|\geq \lambda c(n,s)d^n,
}
	for every $\lambda>0$ for which $0<\big|\{u\geq\lambda\}\cap\Omega'\big|<|\Omega'|$.
	By \eqref{Indra} and \eqref{Grixis} we obtain the first inequality in \eqref{linftysupest}.
	
	In order to prove \eqref{bvlocest} we exploit the uniform perimeter estimate of \cite[Theorem 1.1]{CSV19}. More precisely, 
recalling \eqref{ometa}, we observe that
	\[
	\overline{\Omega'}\subset\bigcup_{x\in\Omega'}B_\frac{d}{2}(x)\subset(\Omega')_\frac{d}{2}\Subset\Omega,
	\]
	hence, since $\overline{\Omega'}$ is compact, there exist $N=N(\Omega',d)$ and $x_1,\dots,x_N\in\Omega'$ such that
	\[
	\overline{\Omega'}\subset\bigcup_{i=1}^N B_\frac{d}{2}(x_i).
	\]
	Since $\{u\geq\lambda\}$ is $s$-minimal in $B_d(x_i)\subset\Omega$ for every $i=1,\dots,N$, we can exploit \cite[Theorem 1.1]{CSV19} to obtain that
	\[
	\Per(\{u\geq\lambda\},\Omega')\leq\sum_{i=1}^N\Per\left(\{u\geq\lambda\},B_\frac{d}{2}(x_i)\right)\leq N(\Omega',d)C(n,s)\left(\frac{d}{2}\right)^{n-1}.
	\]
As a consequence,
recalling \eqref{linftysupest} and exploiting the coarea formula for the $BV$ seminorm, we can now estimate
	\bgs{
|Du|(\Omega')&=\int_{-\|u\|_{L^\infty(\Omega')}}^{\|u\|_{L^\infty(\Omega')}}
\Per(\{u\geq\lambda\},\Omega')\,d\lambda\leq 2\|u\|_{L^\infty(\Omega')}N(\Omega',d)C(n,s)\left(\frac{d}{2}\right)^{n-1}\\
&
\leq \frac{2}{c(n,s)d^n}\|u\|_{L^1(\Omega)}N(\Omega',d)C(n,s)\left(\frac{d}{2}\right)^{n-1},
}
thus concluding the proof of Theorem~\ref{thmreg}.
\end{proof}

\begin{remark}
		In particular, we observe that if $\Ts_s^1(u,\Co\Omega;\cdot)\in L^1(\Omega)$, then it is possible to estimate the $L^\infty$ norm and the $BV$ seminorm purely in terms of the exterior data $u|_{\Co\Omega}$ by exploiting the a priori estimate on the $L^1$ norm given in Lemma \ref{lemmaone}, as indeed
	\[
	\|u\|_{L^\infty(\Omega')}\leq \frac{\mathfrak C_1}{c\dist(\Omega',\partial\Omega)^n}\|\Ts_s^1(u,\Co\Omega;\cdot)\|_{L^1(\Omega)},
	\]
	and
	\[
	|Du|(\Omega')\leq C\mathfrak C_1\|\Ts_s^1(u,\Co\Omega;\cdot)\|_{L^1(\Omega)},
	\]
	for every $\Omega'\Subset\Omega$, where $\mathfrak C_1$, depending only on $n$ and $\Omega$, is defined in Lemma \ref{lemmaone}.

\end{remark}

We turn now to studying the asymptotics as $s\nearrow 1$ of minimizers of the $W^{s,1}$-energy, establishing the relation to functions of least gradient.  We recall that a function $u\in BV(\Omega)$ is said to have \emph{least gradient} in $\Omega$ if
\[
|Du|(\Omega)\leq|Dv|(\Omega)\quad \mbox{for every } v\in BV(\Omega) \mbox{ s.t. spt}(u-v)\Subset\Omega.
\]

Since in this Appendix the parameter $p=1$ is fixed and the domain $\Omega$ may vary, we employ the notation
\[
\En_s(u,\Omega):=\frac{1}{2}\iint_{Q(\Omega)}\frac{|u(x)-u(y)|}{|x-y|^{n+s}}dx\,dy.
\]
Let $\omega_k$ denote the volume of the unit ball in $\R^k$, for $k \ge1$, and $\omega_0:=1$.
Also, given a measurable set $E\subset\R^n$, we denote by
\[
\Per_s^L(E,\Omega):=\frac{1}{2}[\chi_E]_{W^{s,1}(\Omega)}=\int_{E\cap\Omega}\int_{\Omega\setminus E}\frac{dx\,dy}{|x-y|^{n+s}}.
\]
We have the following pointwise convergence result,
which is a consequence of \cite[Theorem 1]{davila} and a suitable geometric argument.
A different proof could be obtained employing the results in \cite{poncespector}.

\begin{theorem}[Pointwise convergence]\label{pwise_th}
	Let $u:\R^n\to\R$ be such that $u|_\Omega\in BV_{loc}(\Omega)$ and
	\eqlab{\label{pwiseconvs1}
	\int_{\R^n}\frac{|u(y)|}{(1+|y|)^{n+\sigma}}\,dy<\infty,
}
	for some $\sigma\in(0,1)$. Then
	\[
	\lim_{s\nearrow1}(1-s)\En_s(u,\Op)=\omega_{n-1}|Du|(\overline{\Op})\quad\mbox{for every }\Op\Subset\Omega\mbox{ with Lipschitz boundary}.
	\]
\end{theorem}

\begin{proof}
	First of all we recall that by \cite[Theorem 1]{davila} we have
	\[
	\lim_{s\nearrow1}(1-s)\frac{1}{2}[u]_{W^{s,1}(\Omega')}=\omega_{n-1}|Du|(\Omega'),
	\]
	for any open set $\Omega'\Subset\Omega$ with Lipschitz boundary. For the computation of the constant $\omega_{n-1}$ see, e.g., \cite[Section 2.2.1]{fractalLuk}.
	
	Now let us fix an open set $\Op\Subset\Omega$ with Lipschitz boundary. In order to compute the limit of the interactions occurring in $\Op\times\Co\Op$ we proceed by adapting the argument in \cite[Section 2.2]{fractalLuk}.
	We first prove that
	\eqlab{\label{eqn}
	\limsup_{s\nearrow1}(1-s)\int_\Op\int_{\Co\Op}\frac{|u(x)-u(y)|}{|x-y|^{n+s}}\,dx\,dy\leq2\omega_{n-1}|Du|(\partial\Op).
}
	
The idea simply consists in splitting appropriately the
domain and exploiting once again \cite[Theorem 1]{davila}. More precisely,
since $\Op$ has Lipschitz boundary, there exists $r_0(\Op)\in(0,\dist(\Op,\partial\Omega))$ such that $\Op_r$ is a bounded open set with Lipschitz boundary for every $r\in(-r_0,r_0)$, where, for $r<0$ we use the notation
	\eqlab{ \label{negs}
	\Op_r:=\{y\in\Op\,|\,\dist(y,\partial\Op)>|r|\},
	}
	recalling that for positive values of $r$ the notation \eqref{ometa} is in place.
	With this notation, for any~$\delta\in(0,r_0)$ we write
	\bgs{
	\int_\Op\int_{\Co\Op}&\frac{|u(x)-u(y)|}{|x-y|^{n+s}}\,dx\,dy=
	\int_{\Op_{-\delta}}\int_{\Co\Op}\frac{|u(x)-u(y)|}{|x-y|^{n+s}}\,dx\,dy+\int_{\Op\setminus\Op_{-\delta}}\int_{\Op_\delta\setminus\Op}\frac{|u(x)-u(y)|}{|x-y|^{n+s}}\,dx\,dy\\
	&\qquad\qquad\qquad\qquad\qquad+\int_{\Op\setminus\Op_{-\delta}}\int_{\Co\Op_\delta}\frac{|u(x)-u(y)|}{|x-y|^{n+s}}\,dx\,dy\\
	&
	\leq\int_{\Op_{-\delta}}\int_{\Co\Op}\frac{|u(x)-u(y)|}{|x-y|^{n+s}}\,dx\,dy+[u]_{W^{s,1}(N_\delta(\partial\Op))}+\int_{\Op}\int_{\Co\Op_\delta}\frac{|u(x)-u(y)|}{|x-y|^{n+s}}\,dx\,dy,
}
	where
	\begin{equation}\label{NOADD}
	N_\delta(\partial\Op):=\{y\in\R^n\,|\,\dist(y,\partial\Op)<\delta\}=\Op_\delta\setminus\overline{\Op_{-\delta}}
	\end{equation}
	is a bounded open set with Lipschitz boundary.
	
	We also observe that there exists a constant $C(\delta,\Op)>0$, depending only on $\Op$ and $\delta$, such that
	\[
	C(\delta,\Op)|x-y|\ge1+|y|,\quad\mbox{for every }(x,y)\in\big(\Op_{-\delta}\times\Co\Op\big)\cup\big(\Op\times\Co\Op_\delta\big).
	\]
Using this, we estimate
	\bgs{
	\int_{\Op_{-\delta}}\int_{\Co\Op}&\frac{|u(x)-u(y)|}{|x-y|^{n+s}}\,dx\,dy
	\leq\int_{\Op_{-\delta}}|u(x)|\left(\int_{\Co\Op}\frac{dy}{|x-y|^{n+s}}\right)dx+\int_{\Op_{-\delta}}\left(\int_{\Co\Op}\frac{|u(y)|}{|x-y|^{n+s}}\,dy\right)dx\\
	&
	\leq\int_{\Op_{-\delta}}|u(x)|\left(\int_{\Co B_\delta(x)}\frac{dy}{|x-y|^{n+s}}\right)dx+C(\delta,\Op)^{n+s}\int_{\Op_{-\delta}}\left(\int_{\Co\Op}\frac{|u(y)|}{(1+|y|)^{n+s}}\,dy\right)dx\\
	&
	\leq\|u\|_{L^1(\Op)}\frac{\Ha^{n-1}(\partial B_1)}{s\delta^s}+C(\delta,\Op)^{n+s}|\Op|\int_{\R^n}\frac{|u(y)|}{(1+|y|)^{n+\sigma}}\,dy,
}
	for every $s\in [\sigma,1)$. Similarly
	\bgs{
		\int_{\Op}\int_{\Co\Op_\delta}\frac{|u(x)-u(y)|}{|x-y|^{n+s}}\,dx\,dy
	\leq\|u\|_{L^1(\Op)}\frac{\Ha^{n-1}(\partial B_1)}{s\delta^s}+C(\delta,\Op)^{n+s}|\Op|\int_{\R^n}\frac{|u(y)|}{(1+|y|)^{n+\sigma}}\,dy.
}	
Notice that
\[
\lim_{s\nearrow1}(1-s)\left(\|u\|_{L^1(\Op)}\frac{\Ha^{n-1}(\partial B_1)}{s\delta^s}+C(\delta,\Op)^{n+s}|\Op|\int_{\R^n}\frac{|u(y)|}{(1+|y|)^{n+\sigma}}\,dy\right)=0.
\]
	Hence, these computations yield
	\begin{equation}\label{esqdrgfhbhfb}\begin{split}
&	\limsup_{s\nearrow1}(1-s)\int_\Op\int_{\Co\Op}\frac{|u(x)-u(y)|}{|x-y|^{n+s}}\,dx\,dy\\&\qquad
\leq\limsup_{s\nearrow1}(1-s)[u]_{W^{s,1}(N_\delta(\partial\Op))}=2\omega_{n-1}|Du|(N_\delta(\partial\Op)),
	\end{split}\end{equation}
	for every $\delta\in(0,r_0)$. Since $N_\delta(\partial\Op)\searrow\partial\Op$ as $\delta\searrow0$ and $|Du|\llcorner\Omega$ is a Radon measure, taking the limit as~$\delta\searrow0$ in~\eqref{esqdrgfhbhfb}
	we obtain \eqref{eqn}.
	
	Notice that the set $\{\delta\in(0,r_0)\,|\,|Du|(\partial\Op_\delta)>0\}$ is at most countable and pick $\delta$ for which $|Du|(\partial\Op_\delta)=0$. We have just proved that, for such a $\delta$,
	\bgs{
\lim_{s\nearrow1}(1-s)\En_s(u,\Op_\delta)=\omega_{n-1}|Du|(\Op_\delta).
}
Now we write
	\bgs{
\En_s(u,\Op_\delta)=\En_s(u,\Op)+\frac{1}{2}[u]_{W^{s,1}(\Op_\delta\setminus\overline{\Op})}+\int_{\Op_\delta\setminus\overline{\Op}}\int_{\Co\Op_\delta}\frac{|u(x)-u(y)|}{|x-y|^{n+s}}\,dx\,dy,
}
	and we remark that by \eqref{eqn} we have
	\[
	\limsup_{s\nearrow1}(1-s)\int_{\Op_\delta\setminus\overline{\Op}}\int_{\Co\Op_\delta}\frac{|u(x)-u(y)|}{|x-y|^{n+s}}\,dx\,dy
	\leq\limsup_{s\nearrow1}(1-s)\int_{\Op_\delta}\int_{\Co\Op_\delta}\frac{|u(x)-u(y)|}{|x-y|^{n+s}}\,dx\,dy
	=0.
	\]
	Thus
	\bgs{
	\lim_{s\nearrow1}(1-s)\En_s(u,\Op)&=\lim_{s\nearrow1}(1-s)\left(\En_s(u,\Op_\delta)-\frac{1}{2}[u]_{W^{s,1}(\Op_\delta\setminus\overline{\Op})}\right)\\
	&
	=\omega_{n-1}\left(|Du|(\Op_\delta)-|Du|(\Op_\delta\setminus\overline{\Op})\right)=\omega_{n-1}|Du|(\overline{\Op}).
}
This concludes the proof of Theorem~\ref{pwise_th}\footnote{{ We point out the following
alternative way to compute the asymptotics, which was suggested by one of the anonymous referees. By the estimates developed in the first part of the proof we have\[
\lim_{s\nearrow1}(1-s)\int_\Op\int_{\Co\Op}\frac{|u(x)-u(y)|}{|x-y|^{n+s}}\,dx\,dy=\lim_{s\nearrow1}(1-s)\int_{\Op\setminus\overline{\Op_{-\delta}}}\int_{\Op_\delta\setminus\overline{\Op}}\frac{|u(x)-u(y)|}{|x-y|^{n+s}}\,dx\,dy.
	\]
Since
\[
\int_{\Op\setminus\overline{\Op_{-\delta}}}\int_{\Op_\delta\setminus\overline{\Op}}\frac{|u(x)-u(y)|}{|x-y|^{n+s}}\,dx\,dy=\frac{1}{2}\left([u]_{W^{s,1}(\Op_\delta\setminus\overline{\Op_{-\delta}})}-[u]_{W^{s,1}(\Op_\delta\setminus\overline{\Op})}-[u]_{W^{s,1}(\Op\setminus\overline{\Op_{-\delta}})}\right),
\]
we obtain
\bgs{
\lim_{s\nearrow1}(1-s)\int_{\Op\setminus\overline{\Op_{-\delta}}}\int_{\Op_\delta\setminus\overline{\Op}}\frac{|u(x)-u(y)|}{|x-y|^{n+s}}\,dx\,dy&
=\omega_{n-1}\left(|Du|(\Op_\delta\setminus\overline{\Op_{-\delta}})-|Du|(\Op_\delta\setminus\overline{\Op})-|Du|(\Op\setminus\overline{\Op_{-\delta}})\right)\\
&=\omega_{n-1}|Du|(\partial\Op),
}
concluding the proof.
}}.
\end{proof}

We investigate  now the $\Gamma$-convergence of the $W^{s,1}$-energy, starting with the following $\Gamma$-liminf inequality.

\begin{proposition}[$\Gamma$-liminf inequality]\label{gammalinfsto1prop}
	Let $u\in L^1_{loc}(\Omega)$, $s_k\nearrow1$ and $u_k\in\W^{s_k,1}(\Omega)$ such that $u_k\to u$ in $L^1_{loc}(\Omega)$. Then
	\eqlab{\label{gammaliminfs1}
	\omega_{n-1} |Du|(\Omega)\leq \liminf_{k\to\infty}(1-s_k)\frac{1}{2}[u_k]_{W^{s_k,1}(\Omega)}\leq \liminf_{k\to\infty}(1-s_k)\En_{s_k}(u_k,\Omega).
}
\end{proposition}

\begin{proof}
	We begin by observing that, since $u_k\to u$ in $L^1_{loc}(\Omega)$, we have
	\[
	\chi_{\{u_k> t\}}\to\chi_{\{u> t\}}\quad\mbox{in }L^1_{loc}(\Omega),
	\]
	for almost every $t\in\R$. Next we recall the generalized coarea formula for the (local part of the) fractional perimeter:
	\[
	\frac{1}{2}[v]_{W^{s,1}(\Omega)}=\int_{-\infty}^\infty \Per_s^L(\{v>t\},\Omega)\,dt,
	\]
	for any measurable function $v:\Omega\to\R$---see, e.g., \cite{visintin}, \cite[Lemma 10]{gammaconv}---and the classical coarea formula:
	\[
	|Dv|(\Omega)=\int_{-\infty}^\infty \Per(\{v>t\},\Omega)\,dt,
	\]
	for any $v\in L^1_{loc}(\Omega)$---see, e.g., \cite[Theorem 3.40]{AFP00}. Now, exploiting the $\Gamma$-liminf inequality for the perimeter functionals in \cite[Theorem 2]{gammaconv} and Fatou's Lemma, we obtain
	\bgs{
	\omega_{n-1} |Du|(\Omega)&=\int_{-\infty}^\infty \omega_{n-1}\Per(\{u>t\},\Omega)\,dt
	\leq\int_{-\infty}^\infty \liminf_{k\to\infty}(1-s_k)\Per_{s_k}^L(\{u_k>t\},\Omega)\,dt\\
	&
	\le \liminf_{k\to\infty}\int_{-\infty}^\infty (1-s_k)\Per_{s_k}^L(\{u_k>t\},\Omega)\,dt
	=\liminf_{k\to\infty}(1-s_k)\frac{1}{2}[u_k]_{W^{s_k,1}(\Omega)}.
}

This proves the first inequality in \eqref{gammaliminfs1}. The second inequality trivially follows by the definition of $\En_{s_k}$.
\end{proof}

The following equi-coercivity property is proved, e.g., in \cite[Section 2]{gammaconv}.

\begin{proposition}[Equi-coercivity]\label{equicoercs1}
	Let $s_k\nearrow1$ and let $u_k\in W^{s_k,1}(\Omega)$ such that
	\[
	\limsup_{k\to\infty}\big(\|u_k\|_{L^1(\Omega')}+(1-s_k)[u_k]_{W^{s_k,1}(\Omega')}\big)<\infty\quad\mbox{for every }\Omega'\Subset\Omega.
	\]
	Then $\{u_k\}_k$ is relatively compact in $L^1_{loc}(\Omega)$ and any limit point $u$ belongs to $BV_{loc}(\Omega)$.	
\end{proposition}

We now prove the convergence of $(s,1)$-minimizers to functions of least gradient, exploiting the results of \cite{gammaconv}, which hold for sets.

\begin{theorem}[Convergence of minimizers]\label{conv_min_sto1}
	Let $s_k\nearrow1$ and $\varphi_k:\Co\Omega\to\R$ be such that
	\eqlab{\label{bded_data_hp}
	\limsup_{k\to\infty}\|\varphi_k\|_{L^\infty(\Omega_R\setminus\Omega)}<\infty,
}
	for some $R=R(n,\Omega)>0$ big enough, and
	\eqlab{\label{tail_bound_sto1}
	\limsup_{k\to\infty}(1-s_k)\int_{\Co\Omega_R}\frac{|\varphi_k(y)|}{(1+|y|)^{n+s_k}}\,dy<\infty.
}
	Let $u_k\in\W^{s_k,1}_{\varphi_k}(\Omega)$ be any sequence of $(s_k,1)$-minimizers. Then, there exist $u\in L^\infty(\Omega)\cap BV(\Omega)$ of least gradient in $\Omega$ and a subsequence $s_{k_h}\nearrow1$ such that $u_{k_h}\to u$ in $L^1(\Omega)$.
	
	Moreover, if
	\eqlab{\label{tail_vanish_sto1}
	\lim_{k\to\infty}(1-s_k)\int_{\Co\Omega_R}\frac{|\varphi_k(y)|}{(1+|y|)^{n+s_k}}\,dy=0,
}
	then
	\eqlab{\label{senergy_conv_to1}
	\lim_{h\to\infty}(1-s_{k_h})\En_{s_{k_h}}(u_{k_h},\Op)=\omega_{n-1}|Du|(\Op),
}
	whenever $\Op\Subset\Omega$ is an open set with Lipschitz boundary such that $|Du|(\partial\Op)=0$.
\end{theorem}

\begin{proof}
	We begin by observing that by \eqref{bded_data_hp} we can assume without loss of generality that
	\[
	\|\varphi_k\|_{L^\infty(\Omega_R\setminus\Omega)}\leq M,
	\]
	for every $k\in\N$, for some $M>0$. Next we prove that
	\eqlab{\label{tailsto1unifbd}
	\limsup_{k\to\infty}(1-s_k)\|\Ts_{s_k}^1(\varphi,\Co\Omega;\cdot)\|_{L^1(\Omega)}<\infty.
}
	Indeed, we observe at first that there exists a constant $C=C(\Omega,R)>1$ such that
	\eqlab{\label{usefulestytail}
	C|x-y|\geq 1+|y|\quad\mbox{for every }(x,y)\in\Omega\times\Co\Omega_R.
}
	Hence we can estimate
	\eqlab{\label{Calix}
	\|\Ts_{s_k}^1(\varphi,\Co\Omega;\cdot)\|_{L^1(\Omega)}&=\int_\Omega\left(\int_{\Omega_R\setminus\Omega}\frac{|\varphi_k(y)|}{|x-y|^{n+s_k}}\,dy\right)dx+\int_\Omega\left(\int_{\Co\Omega_R}\frac{|\varphi_k(y)|}{|x-y|^{n+s_k}}\,dy\right)dx\\
	&
	\leq M\int_\Omega\int_{\Co\Omega}\frac{dx\,dy}{|x-y|^{n+s_k}}+C^{n+s_k}\int_\Omega\left(\int_{\Co\Omega_R}\frac{|\varphi_k(y)|}{(1+|y|)^{n+s_k}}\,dy\right)dx\\
	&
	\leq M\Per_{s_k}(\Omega,\R^n)+C^{n+1}|\Omega|\int_{\Co\Omega_R}\frac{|\varphi_k(y)|}{(1+|y|)^{n+s_k}}\,dy.
}
	Since $\Omega$ is a bounded open set with Lipschitz boundary, we have
	\[
	\lim_{k\to\infty}(1-s_k)\Per_{s_k}(\Omega,\R^n)=\omega_{n-1}\Ha^{n-1}(\partial\Omega).
	\]
	Together with \eqref{Calix} and \eqref{tail_bound_sto1}, this concludes the proof of \eqref{tailsto1unifbd}.
	
	Thus, by the a priori estimate \eqref{never_enough_remarks} we have
	\eqlab{\label{unifest}
		\limsup_{k\to\infty}(1-s_k)\|u_k\|_{W^{s_k,1}(\Omega)}\leq\mathfrak C_1 \limsup_{k\to\infty}(1-s_k)\|\Ts_{s_k}^1(\varphi,\Co\Omega;\cdot)\|_{L^1(\Omega)}<\infty.
	}	
Moreover, we recall that by \cite[Theorem 4.4]{bdlv20} we have that there exists $\Theta(n,s_k)>1$ such that 
\[
\|u_k\|_{L^\infty(\Omega)}\leq\|\varphi_k\|_{L^\infty(\Omega_{\Theta(n,s_k)\diam(\Omega)}\setminus\Omega)}.
\]
A careful inspection of the proof of \cite[Theorem 4.1]{bdlv20}---in particular of the last equation in display in the proof---reveals that $\Theta(n,s)>1$ can be chosen in such a way that $1<\Theta(n,s)\leq\Theta(n, 1/2)$ for every $s\in [1/2,1)$. Since $s_k\nearrow1$, we can suppose without loss of generality that $s_k>1/2$ for every $k$, hence, setting $R=R(n,\Omega):=\Theta(n, 1/2)\diam(\Omega)$, we obtain
\eqlab{\label{Izzet}
	\|u_k\|_{L^\infty(\Omega)}\leq\|\varphi_k\|_{L^\infty(\Omega_R\setminus\Omega)}\leq M
}
for every $k$.

By \eqref{unifest} and \eqref{Izzet}, we can apply Proposition \ref{equicoercs1} to obtain that---up to a subsequence that we relabel for simplicity---there exists a function $u\in BV_{loc}(\Omega)$ such that $u_k\to u$ in $L^1_{loc}(\Omega)$. Actually, by \eqref{Izzet} we have that $u\in L^\infty(\Omega)$ and $u_k\to u$ in $L^1(\Omega)$. Moreover \eqref{gammaliminfs1} and \eqref{unifest} ensure that
	\[
	\omega_{n-1}|Du|(\Omega)\leq\liminf_{k\to\infty}(1-s_k)\frac{1}{2}[u_k]_{W^{s_k,1}(\Omega)}<\infty,
	\]
	hence $u\in BV(\Omega)$.
	
	We proceed now to prove that $u$ is a function of least gradient in $\Omega$.
	In order to do this, notice that the convergence $u_k\to u$ in $L^1(\Omega)$ implies that there exists a set $\Sigma\subset\R$ such that $|\Sigma|=0$ and
	\[
	\chi_{\{u_k\geq\lambda\}}\xrightarrow{k\to\infty}\chi_{\{u\geq\lambda\}}\quad\mbox{in }L^1(\Omega),
	\]
	for every $\lambda\in\R\setminus\Sigma$. By \cite[Theorem 1.3]{bdlv20} we know that the level set $\{u_k\geq\lambda\}$ is $s_k$-minimal in $\Omega$ for every $\lambda\in\R$ and $k\in\N$. Thus \cite[Theorem 3]{gammaconv} ensures that the level set $\{u\geq\lambda\}$ is a local minimizer of $\Per(\,\cdot\,,\Omega)$ for every $\lambda\in\R\setminus\Sigma$---indeed we observe that the assumption $\chi_{E_i}\to\chi_E$ in $L^1_{loc}(\R^n)$ in \cite[Theorem 3]{gammaconv} is not really needed, since it suffices to assume the convergence in $L^1(\Omega)$ and to exploit the uniform global boundedness of the characteristic functions.
	
	Consider now $v\in BV(\Omega)$ such that $\mbox{spt}(u-v)\Subset\Omega$. Since $\mbox{spt}(\chi_{\{u\geq\lambda\}}-\chi_{\{v\geq\lambda\}})\Subset\Omega$, by local minimality we have that
	\[
	\Per(\{u\geq\lambda\},\Omega)\leq \Per(\{v\geq\lambda\},\Omega),
	\]
	for every $\lambda\in\R\setminus\Sigma$. Hence, by the coarea formula for the $BV$ seminorm,
	\bgs{
|Du|(\Omega)&=\int_\R\Per(\{u\geq\lambda\},\Omega)\,d\lambda	=\int_{\R\setminus\Sigma}\Per(\{u\geq\lambda\},\Omega)\,d\lambda\\
&
\leq \int_{\R\setminus\Sigma}\Per(\{v\geq\lambda\},\Omega)\,d\lambda=\int_\R\Per(\{v\geq\lambda\},\Omega)\,d\lambda=|Dv|(\Omega),
}
	proving that $u$ is a function of least gradient in $\Omega$.
	
	\medskip
	
	It remains to prove the convergence of the energies \eqref{senergy_conv_to1} under the hypothesis \eqref{tail_vanish_sto1}, for which we adapt the argument of the proof of \cite[Theorem 3]{gammaconv}.
	
	We consider the monotone set functions $\alpha_k(\Op):=(1-s_k)\frac{1}{2}[u_k]_{W^{s_k,1}(\Op)}$ for every open set $\Op\subset\Omega$, extended to
	\[
	\alpha_k(B):=\inf\left\{\alpha_k(\Op)\,|\,B\subset\Op\subset\Omega,\,\Op\mbox{ open}\right\},
	\]
	for every Borel set $B\subset\Omega$. Each $\alpha_k$ is a monotone, regular and super-additive set function in the sense of \cite[Section 5.2]{gammaconv}. By \eqref{unifest} and \cite[Theorem 21]{gammaconv}, up to extracting a subsequence that we relabel for simplicity, $\alpha_k$ weakly converges to a regular, monotone and super-additive set function $\alpha$, as $k\to\infty$.
	
	Let $\Op\Subset\Omega$ be an open set with Lipschitz boundary such that $\alpha(\partial\Op)=0$. We remark that there exists $r_0:=r_0(\Op)\in(0,\dist(\Op,\partial\Omega))$ small enough such that $\Op_r$ has Lipschitz boundary for every $r\in(-r_0,r_0)$.
	
Let us fix $\varrho\in(0,r_0/3)$ and consider a function $\psi\in C^\infty_c(\R^n)$ such that $0\leq\psi\leq1$, $\psi\equiv1$ in $\Op_{-2\varrho}$, $\psi\equiv0$ in $\Co\Op_{-\varrho}$, and $|\nabla\psi|\leq 2/\varrho$ (we recall the notation \eqref{negs}). We define the function $v_k:\R^n\to\R$ by setting $v_k:=\psi u+(1-\psi)u_k$. Since $v_k=u_k$ in $\Co\Op_{-\varrho}$, by the minimality of $u_k$ we have
	\[
	\En_{s_k}(u_k,\Op)\leq\En_{s_k}(v_k,\Op).
	\]
	
	We now estimate the energy of $v_k$, beginning with the contributions occurring inside of $\Op$. Arguing as in the proof of \cite[Proposition 11]{gammaconv}---see in particular formula (27) there---and taking into account the fact that $\|u\|_{L^\infty(\Omega)},\,\|u_k\|_{L^\infty(\Omega)}\leq M$, we obtain
	\bgs{
		\,[v_k]_{W^{s_k,1}(\Op)}&\leq[u]_{W^{s_k,1}(\Op)}+[u_k]_{W^{s_k,1}(\Op\setminus\overline{\Op_{-3\varrho}})}+C(\Op,\varrho)\frac{\|u-u_k\|_{L^1(\Op_{-\varrho}\setminus\Op_{-2\varrho})}}{1-s_k}\\
		&
		\qquad+C(\Op,\varrho)\|u-u_k\|_{L^1(\Op)}+\frac{M\,C(\Op)}{\varrho^{n+s_k}}.
	}
	Thus, by \cite[Theorem 1]{davila}, and since $u_k\to u$ in $L^1(\Omega)$, we have
	\eqlab{\label{WWallace}
\limsup_{k\to\infty}(1-s_k)\frac{1}{2}	[v_k]_{W^{s_k,1}(\Op)}
\leq\omega_{n-1}|Du|(\Op)+\limsup_{k\to\infty}\alpha_k(\Op\setminus\overline{\Op_{-3\varrho}}).
}
	On the other hand, we can write
	\bgs{
\int_\Op\int_{\Co\Op}&\frac{|v_k(x)-v_k(y)|}{|x-y|^{n+s_k}}\,dx\,dy\\
&
=\int_{\Op_{-\varrho}}\int_{\Co\Op}\frac{|v_k(x)-u_k(y)|}{|x-y|^{n+s_k}}\,dx\,dy+\int_{\Op\setminus\Op_{-\varrho}}\int_{\Co\Op}\frac{|u_k(x)-u_k(y)|}{|x-y|^{n+s_k}}\,dx\,dy\\&
=:I+II.
}
	Then, recalling \eqref{usefulestytail}, we can estimate
	\bgs{
	I&=\int_{\Op_{-\varrho}}\int_{\Co\Op}\frac{|v_k(x)-u_k(y)|}{|x-y|^{n+s_k}}\,dx\,dy\\
	&
	=\int_{\Op_{-\varrho}}\int_{\Omega_R\setminus\Op}\frac{|v_k(x)-u_k(y)|}{|x-y|^{n+s_k}}\,dx\,dy+\int_{\Op_{-\varrho}}\int_{\Co\Omega_R}\frac{|v_k(x)-u_k(y)|}{|x-y|^{n+s_k}}\,dx\,dy\\
	&
	\leq 2M\int_{\Op_{-\varrho}}\int_{\Omega_R\setminus\Op}\frac{dx\,dy}{|x-y|^{n+s_k}}
	+\int_{\Op_{-\varrho}}|v_k(x)|\left(\int_{\Co\Omega_R}\frac{dy}{|x-y|^{n+s_k}}\right)dx\\
	&\qquad\qquad
	+C(\Omega,R)^{n+s_k}\int_{\Op_{-\varrho}}\int_{\Co\Omega_R}\frac{|u_k(y)|}{(1+|y|)^{n+s_k}}\,dx\,dy\\
	&
	\leq 3M|\Op_{-\varrho}|\frac{\Ha^{n-1}(\partial B_1)}{s_k\varrho^{s_k}}
	+C(\Omega,R)^{n+1}|\Op_{-\varrho}|\int_{\Co\Omega_R}\frac{|\varphi_k(y)|}{(1+|y|)^{n+s_k}}\,dy.
}
As for the contribution $II$, arguing similarly we have
\bgs{
II&=
\int_{\Op\setminus\Op_{-\varrho}}\int_{\Co\Op}\frac{|u_k(x)-u_k(y)|}{|x-y|^{n+s_k}}\,dx\,dy\\
&
=\int_{\Op\setminus\Op_{-\varrho}}\int_{\Op_\varrho\setminus\Op}\frac{|u_k(x)-u_k(y)|}{|x-y|^{n+s_k}}\,dx\,dy
+\int_{\Op\setminus\Op_{-\varrho}}\int_{\Co\Op_\varrho}\frac{|u_k(x)-u_k(y)|}{|x-y|^{n+s_k}}\,dx\,dy\\
&
\leq[u_k]_{W^{s_k,1}(\Op_\varrho\setminus\overline{\Op_{-\varrho}})}
+2M\int_{\Op\setminus\Op_{-\varrho}}\int_{\Omega_R\setminus\Op_\varrho}\frac{dx\,dy}{|x-y|^{n+s_k}}+\int_{\Op\setminus\Op_{-\varrho}}|u_k(x)|\left(\int_{\Co\Omega_R}\frac{dy}{|x-y|^{n+s_k}}\right)dx\\
&\qquad\qquad+C(\Omega,R)^{n+s_k}\int_{\Op\setminus\Op_{-\varrho}}\int_{\Co\Omega_R}\frac{|u_k(y)|}{(1+|y|)^{n+s_k}}\,dx\,dy\\
&
\leq[u_k]_{W^{s_k,1}(\Op_\varrho\setminus\overline{\Op_{-\varrho}})}
+3M|\Op\setminus\Op_{-\varrho}|\frac{\Ha^{n-1}(\partial B_1)}{s_k\varrho^{s_k}}
+C(\Omega,R)^{n+1}|\Op\setminus\Op_{-\varrho}|\int_{\Co\Omega_R}\frac{|\varphi_k(y)|}{(1+|y|)^{n+s_k}}\,dy.
}
Therefore, exploiting \eqref{tail_vanish_sto1} we obtain
\eqlab{\label{RRoy}
\limsup_{k\to\infty}(1-s_k)\int_\Op\int_{\Co\Op}\frac{|v_k(x)-v_k(y)|}{|x-y|^{n+s_k}}\,dx\,dy
\leq 2\limsup_{k\to\infty}\alpha_k(\Op_\varrho\setminus\overline{\Op_{-\varrho}}).
}
Thus, by Proposition \ref{gammalinfsto1prop}
and the minimality of $u_k$, and exploiting \eqref{WWallace} and \eqref{RRoy}, we conclude that
\eqlab{\label{last_eq}
\omega_{n-1}|Du|(\Op)&\leq\limsup_{k\to\infty}(1-s_k)\En_{s_k}(u_k,\Op)\leq\limsup_{k\to\infty}(1-s_k)\En_{s_k}(v_k,\Op)\\
&
\leq\omega_{n-1}|Du|(\Op)+3\limsup_{k\to\infty}\alpha_k(\Op_\varrho\setminus\overline{\Op_{-3\varrho}}),
}
for every $\varrho\in(0,r_0/3)$.
Since $\alpha(\partial\Op)=0$, by \cite[Proposition 22]{gammaconv} we have that
\[
\lim_{\varrho\searrow0}\limsup_{k\to\infty}\alpha_k(\Op_\varrho\setminus\overline{\Op_{-3\varrho}})=0.
\]
In light of~\eqref{last_eq}, this proves \eqref{senergy_conv_to1} in every open set $\Op\Subset\Omega$ with Lipschitz boundary, such that~$\alpha(\partial\Op)=0$.

Hence, in order to conclude the proof of
Theorem~\ref{conv_min_sto1}, we are left to show that $|Du|(\partial\Op)=0$ implies that~$\alpha(\partial\Op)=0$.
For this, we first observe that, by definition of weak convergence of
monotone set functions, we have in particular that
\[
\alpha(\Omega')\leq\liminf_{k\to\infty}\alpha_k(\Omega'),
\]
for every open set $\Omega'\subset\Omega$. Therefore, if $\Omega'\Subset\Omega$ is an open set with Lipschitz boundary, such that~$\alpha(\partial\Omega')=0$,
by \eqref{senergy_conv_to1} we obtain that
\eqlab{\label{Is_this_finally_the_end}
\alpha(\Omega')\leq\liminf_{k\to\infty}\alpha_k(\Omega')\leq\limsup_{k\to\infty}(1-s_k)\En_{s_k}(u_k,\Omega')=\omega_{n-1}|Du|(\Omega').
}
Let us now consider an open set $\Op\Subset\Omega$ with Lipschitz boundary, such that $|Du|(\partial\Op)=0$. We point out that, since $\alpha$ is locally finite in $\Omega$, as a consequence of the super-additivity and
the monotonicity of $\alpha$, the set
\[
\Sigma_\alpha(\partial\Op):=\left\{\delta\in(0,r_0)\,|\,\alpha\big(\partial( \Op_\delta\setminus\overline{\Op_{-\delta}})\big)>0\right\}
\]
is at most countable. Therefore, we can find $\delta_h\searrow0$ such that $\Op_{\delta_h}\setminus\overline{\Op_{-\delta_h}}$ is a bounded open set with Lipschitz boundary, with $\alpha\big(\partial(\Op_{\delta_h}\setminus\overline{\Op_{-\delta_h}})\big)=0$ for every $h$. By the monotonicity of the set function $\alpha$, exploiting \eqref{Is_this_finally_the_end} and the fact that $|Du|\llcorner\Omega$ is a Radon measure,
we obtain that
\bgs{
\alpha(\partial\Op)\leq\limsup_{h\to\infty}\alpha\big(\Op_{\delta_h}\setminus\overline{\Op_{-\delta_h}}\big)\leq\limsup_{h\to\infty}\omega_{n-1}|Du|\big(\Op_{\delta_h}\setminus\overline{\Op_{-\delta_h}}\big)=\omega_{n-1}|Du|(\partial\Op)=0.
}
This concludes the proof of Theorem~\ref{conv_min_sto1}.
\end{proof}

We stress that in order to obtain the convergence of the energies in~\eqref{senergy_conv_to1} it is necessary to make the stronger assumption \eqref{tail_vanish_sto1} in place of \eqref{tail_bound_sto1}, as shown by the following example.

\begin{example}
	Let $R=R(n,\Omega)$ be as in Theorem \ref{conv_min_sto1} and
	consider the function $\varphi_s:\R^n\to\R$ defined by
	$\varphi_s:=\frac{1}{1-s}\chi_{\Co\Omega_R}$.
	By \eqref{Izzet}, since here $\varphi_s=0$ in $\Omega_R\setminus \Omega$, we have that~$\varphi_s$ is the unique minimizer in $\W^{s,1}_{\varphi_s}(\Omega)$ for every $s\in[1/2,1)$. Clearly, as $s\nearrow1$ the functions $\varphi_s$ converge in $L^1(\Omega)$ to the function $u\equiv0$, which is of least gradient in $\Omega$.
	We observe that $\varphi_s$ satisfies \eqref{tail_bound_sto1},
	as indeed
	\bgs{
	(1-s)\int_{\Co\Omega_R}\frac{|\varphi_s(y)|}{(1+|y|)^{n+s}}\,dy=\int_{\Co\Omega_R}\frac{dy}{(1+|y|)^{n+s}}
	\leq \int_{\Co\Omega_R}\frac{dy}{(1+|y|)^{n+\frac{1}{2}}}<\infty,
}
	for every $s\in[1/2,1)$. On the other hand,
	given any $\Op\Subset\Omega$ we have that
	\bgs{
\liminf_{s\to1}(1-s)\En_s(\varphi_s,\Op)=\liminf_{s\to1}\int_\Op\int_{\Co\Omega_R}\frac{dx\,dy}{|x-y|^{n+s}}\geq c>0,
}
	proving that the convergence of the energies \eqref{senergy_conv_to1} can not hold true, since $|Du|(\Op)=0$.
\end{example}

For the sake of completeness, we provide also a proof of the $\Gamma$-limsup inequality, in the case in which $\partial\Omega$ is of class $C^2$.

\begin{theorem}[$\Gamma$-limsup inequality]\label{qwertyuilhkjghfgdcv}
	Let $\Omega\subset\R^n$ be a bounded open set with $C^2$ boundary, $u\in L^1_{loc}(\R^n)$ such that $u|_\Omega\in BV(\Omega)$ and let $s_k\nearrow1$. Then, there exists a sequence $u_k\in\W^{s_k,1}(\Omega)$ such that $u_k\to u$ in $L^1_{loc}(\R^n)$ and
	\[
	\lim_{k\to\infty}(1-s_k)\En_{s_k}(u_k,\Omega)=\omega_{n-1}|Du|(\Omega).
	\]
\end{theorem}

\begin{proof}
	The main difficulty of the proof resides in properly approximating $u$ around $\partial\Omega$. In order to do this, we exploit the signed distance function $\bar{d}_\Omega$, which is defined as
	\[
	\bar{d}_\Omega(x):=\dist(x,\Omega)-\dist(x,\Co\Omega),
	\]
	for every $x\in\R^n$. For the properties of the signed
	distance function that we employ here, we refer
	to \cite[Appendix B.1]{tesilu} and the references cited therein.
	Since $\Omega$ is bounded and has $C^2$ boundary,
	there exists $r_0(\Omega)>0$ such that
	$\bar{d}_\Omega\in C^2(N_{2r_0}(\partial\Omega))$,
	where we use the notation introduced in~\eqref{NOADD}.
	For any $\delta\in(0,r_0)$ we consider the
	projection $\pi_\delta:\Omega_\delta\setminus\Omega_{-\delta}
	\twoheadrightarrow\partial\Omega_{-\delta}$ defined by
	\[
	\pi_\delta(x):=x-(\delta+\bar{d}_\Omega(x))\nabla\bar{d}_\Omega(x),
	\]
	and we observe that $\pi_\delta\in C^1(\Omega_\delta\setminus\Omega_{-\delta},\R^n)$---see, e.g., \cite[Proposition B.1.6]{tesilu} for related computations---with
	\[
	D\pi_\delta(x)=\textrm{Id}_n-(\delta+\bar{d}_\Omega(x))D^2\bar{d}_\Omega(x)-\nabla\bar{d}_\Omega(x)\otimes\nabla\bar{d}_\Omega(x).
	\]
	In particular,
	\eqlab{\label{eqn1}
	|D\pi_\delta(x)|\leq C_1(\Omega)\quad\mbox{for every }x\in\Omega_\delta\setminus\Omega_{-\delta},
}
	for some constant $C_1(\Omega)>0$ that does not depend on $\delta\in(0,r_0)$.
	For every $\varrho\in(0,\delta)$ we consider also the function $\Phi_\varrho:\partial\Omega_{-\delta}\to\partial\Omega_{-\varrho}$ defined by
	\[
	\Phi_\varrho(x):=x+(\delta-\varrho)\nabla\bar{d}_\Omega(x),
	\]
	which is a bijection of class $C^1$. We observe that
	\eqlab{\label{eqn6}
	\pi_\delta(\Phi_{\varrho}(x))&=\Phi_\varrho(x)-(\delta+\bar{d}_\Omega(\Phi_\varrho(x)))\nabla\bar{d}_\Omega(\Phi_\varrho(x))\\
	&
	=x+(\delta-\varrho)\nabla\bar{d}_\Omega(x)-(\delta-\varrho)\nabla\bar{d}_\Omega(x+(\delta-\varrho)\nabla\bar{d}_\Omega(x))\\
	&
	=x+(\delta-\varrho)\nabla\bar{d}_\Omega(x)
	-(\delta-\varrho)\nabla\bar{d}_\Omega(x)
	\\&=x,
}
	for every $x\in\partial\Omega_{-\delta}$, and
	\eqlab{\label{eqn2}
	\sup_{x\in\partial\Omega_{-\delta}}|D\Phi_\varrho(x)|\leq C_2(\Omega),
}
	for some constant $C_2(\Omega)>0$
	that is independent of $\delta\in(0,r_0)$ and $\varrho\in(0,\delta)$.
	
	With these preliminaries at hand, we recall that, since $u\in BV(\Omega)$, given $\eps>0$ there exists $v\in C^\infty(\Omega)\cap BV(\Omega)$ such that
	\[
	\|u-v\|_{L^1(\Omega)}+\big||Du|(\Omega)-|Dv|(\Omega)\big|<\eps.
	\]
	Notice that we can find $\delta_0\in(0,r_0)$ small enough such that
	\[
	\int_{\Omega\setminus\Omega_{-\delta_0}}|v(x)|\,dx+|Dv|(\Omega\setminus\Omega_{-\delta_0})=\int_{\Omega\setminus\Omega_{-\delta_0}}\big(|u(x)|+|\nabla v(x)|\big)\,dx<\eps.
	\]
	Now, by using the coarea formula for $\bar{d}_\Omega$ we can write
		\[
	\int_{\Omega\setminus\Omega_{-\delta_0}}\big(|u(x)|+|\nabla v(x)|\big)\,dx
	=\int_{-\delta_0}^0\left(\int_{\{\bar{d}_\Omega=r\}}\big(|u(x)|+|\nabla v(x)|\big)\,d\Ha^{n-1}_x\right)dr.
	\]
	This implies that we can find a sequence $\{\delta_h\}_h\subset(0,\delta_0)$ such that $\delta_h\searrow0$ and
	\eqlab{\label{eqn3}
	\int_{\{\bar{d}_\Omega=-\delta_h\}}\big(|u(x)|+|\nabla v(x)|\big)\,d\Ha^{n-1}_x<\frac{\eps}{\delta_h},
}
	for every $h\in\N$. Now, for any~$h\in\N$ we define the function $v_h:\R^n\to\R$ by setting
	\[
	v_h:=v\chi_{\Omega_{-\delta_h}}+(v\circ\pi_{\delta_h})\chi_{\Omega_{\delta_h}\setminus\Omega_{-\delta_h}}
	+\min\{1/\delta_h,\max\{-1/\delta_h,u\}\}\chi_{\Co\Omega_{\delta_h}}.
	\]
	We observe that $v_h\in C^0(\Omega_{\delta_h})\cap BV(\Omega_{\delta_h})\cap C^1(\Omega_{\delta_h}\setminus\partial\Omega_{-\delta_h})$, hence
	\[
	|Dv_h|(\partial\Omega_{-\delta_h})=0,\qquad|Dv_h|(\partial\Omega)=0,
	\]
	and
	\bgs{
|Dv_h|(\Omega)&=\int_{\Omega_{-\delta_h}}|\nabla v_h(x)|\,dx+\int_{\Omega\setminus\overline{\Omega_{-\delta_h}}}|\nabla v_h(x)|\,dx\\
&
=\int_{\Omega_{-\delta_h}}|\nabla v(x)|\,dx+\int_{\Omega\setminus\overline{\Omega_{-\delta_h}}}|D\pi_{\delta_h}(x)\cdot(\nabla v)(\pi_{\delta_h}(x))|\,dx.
}
By using the coarea formula for $\bar{d}_\Omega$, exploiting \eqref{eqn1}, \eqref{eqn6}, \eqref{eqn2}, and \eqref{eqn3}, we can estimate
\bgs{
\int_{\Omega\setminus\overline{\Omega_{-\delta_h}}}|D\pi_{\delta_h}(x)&\cdot(\nabla v)(\pi_{\delta_h}(x))|\,dx
\leq C_1(\Omega)\int_{\Omega\setminus\overline{\Omega_{-\delta_h}}}|(\nabla v)(\pi_{\delta_h}(x))|\,dx\\
&
=C_1(\Omega)\int_{-\delta_h}^0\left(\int_{\partial\Omega_{\varrho}}|(\nabla v)(\pi_{\delta_h}(x))|\,d\Ha^{n-1}_x\right)d\varrho\\
&
=C_1(\Omega)\int_{-\delta_h}^0\left(\int_{\Phi_{-\varrho}(\partial\Omega_{-\delta_h})}|(\nabla v)(\pi_{\delta_h}(x))|\,d\Ha^{n-1}_x\right)d\varrho\\
&
\leq C(\Omega)\int_{-\delta_h}^0\left(\int_{\partial\Omega_{-\delta_h}}|\nabla v(y)|\,d\Ha^{n-1}_y\right)d\varrho\\
&
\leq C(\Omega)\frac{\delta_h\eps}{\delta_h}=C(\Omega)\eps.
}
Moreover, we observe that
\[
\left|\int_{\Omega_{-\delta_h}}|\nabla v(x)|\,dx-|Dv|(\Omega)\right|=\int_{\Omega\setminus\Omega_{-\delta_h}}|\nabla v(x)|\,dx\leq\int_{\Omega\setminus\Omega_{-\delta_0}}|\nabla v(x)|\,dx<\eps,
\]
for every $h$.
Therefore, we obtain that
\bgs{
\big||Dv_h|(\Omega)-|Du|(\Omega)\big|\leq\big||Dv_h|(\Omega)-|Dv|(\Omega)\big|+\big||Dv|(\Omega)-|Du|(\Omega)\big|<\eps+C(\Omega)\eps+\eps,
}
for every $h$.

A similar argument yields that
\bgs{
\|u-v_h\|_{L^1(\Omega_{\delta_h})}<(2+C(\Omega))\eps,
}
for every $h$. On the other hand, we remark that, for every fixed
$R>0$ such that $\Omega\Subset B_R$ we have,
by Lebesgue's Dominated Convergence Theorem, that
\bgs{
\lim_{h\to\infty}\int_{B_R\setminus\Omega_{\delta_h}}|v_h-u|\,dx=0.
}

	By considering a sequence $\eps_\ell\searrow0$ in the above computations, we have just proved that we can find a sequence $\delta_\ell\searrow0$ and a sequence of functions $w_\ell:\R^n\to\R$ such that $w_\ell\in BV(\Omega_{\delta_\ell})$, with $|Dw_\ell|(\partial\Omega)=0$, and $w_\ell\in L^\infty(\R^n)$, such that
	\[
	w_\ell\to u\quad\mbox{in }L^1_{loc}(\R^n)\qquad\mbox{ and }\qquad
	\lim_{\ell\to\infty}|Dw_\ell|(\Omega)=|Du|(\Omega).
	\]
	For any such function $w_\ell$ we can apply Theorem \ref{pwise_th}, with $\Omega_{\delta_\ell}$ in place of $\Omega$ and $\Omega$ in place of $\Op$, to obtain
	that
	\[
	\lim_{k\to\infty}(1-s_k)\En_{s_k}(w_\ell,\Omega)=\omega_{n-1}|Dw_\ell|(\Omega).
	\]
	The conclusion of Theorem~\ref{qwertyuilhkjghfgdcv}
	then follows by a standard diagonal argument---similar to the one employed in the proofs of
	Theorems~\ref{gammy} and~\ref{fixeddata}.
\end{proof}

\bibliography{biblio}
\bibliographystyle{plain}

\end{document}